\documentclass[11pt, oneside]{article}   	
\usepackage{geometry}                		
\usepackage[parfill]{parskip}    			
\usepackage{graphicx}				
\usepackage{amssymb}
\usepackage{mathtools}
\usepackage{enumerate}
\usepackage{tikz}
\usepackage{amsmath,subeqnarray}
\usepackage{eqsection, indent}

\usepackage[utf8]{inputenc}
\usepackage[english]{babel}
\usepackage[scr]{rsfso}
\usepackage{mathtools}
\usepackage{amsthm}
\usepackage{bm}

\usetikzlibrary{matrix}
\usetikzlibrary{cd}
\usepackage{forest}

\usepackage{hyperref}

\newtheorem{theorem}{Theorem}[section]
\newtheorem{proposition}[theorem]{Proposition}
\newtheorem{lemma}[theorem]{Lemma}
\newtheorem{corollary}[theorem]{Corollary}
\newtheorem{conjecture}[theorem]{Conjecture}

\theoremstyle{definition}
\newtheorem{definition}[theorem]{Definition}

\newtheorem{remark}[theorem]{Remark}

\def\reff#1{(\protect\ref{#1})}

\newcommand*{\Scale}[2][4]{\scalebox{#1}{$#2$}}

\title{Cyclic sieving phenomena \\
via combinatorics of continued fractions}
\author{Bishal Deb${}^{1}$\\[5mm]
${}^1$Yau Mathematical Sciences Center, Tsinghua University, Beijing 100084, China
}
\date{August 17, 2025}

\begin{document}

\maketitle

\begin{abstract}
We will exhibit several instances of the cyclic sieving phenomenon involving statistics and involutions on the following combinatorial families of objects: 
permutations, set partitions, perfect matchings, D-permutations (and its subclasses).
Our results will be based on continued fraction identities enumerating these objects.
Our instances of cyclic sieving phenomenon for permutations involve the Corteel involution;
this was first studied by Adams, Elder, Lafreni\`ere, McNicholas, Striker and Welch (arxiv~2024).
We will reprove several of their results using our setting of continued fractions;
we will also prove two of their conjectures.
Our study of set partitions and perfect matchings will involve the
Kasraoui--Zeng involution 
and the Chen--Deng--Du--Stanley--Yan (CDDSY) involution.
Finally, for D-permutations, we will construct a new involution 
which we call the Genocchi--Corteel involution.
The common feature of all of these involutions, other than the CDDSY involution,
is that they are constructed via bijections to weighted lattice paths, 
and that they exchange crossings and nestings on their respective objects.

\end{abstract}


\medskip
\noindent
{\bf Mathematics Subject Classification (MSC 2020) codes:}
05A19 (Primary);\\
05A05,
05A10,
05A15,
05A18,
05E18,
30B70

\vspace*{1cm}

\newcommand{\wexx}{{\rm wexx}}

\newcommand{\cdes}{{\rm cdes}}

\newcommand{\eqdef}{\stackrel{\rm def}{=}}

\newcommand{\N}{\mathbb{N}}
\newcommand{\Z}{\mathbb{Z}}
\newcommand{\R}{\mathbb{R}}
\newcommand{\Q}{\mathbb{Q}}

\newcommand{\wt}{{\rm wt}}

\newcommand{\proofof}[1]{\bigskip\noindent{\sc Proof of #1.\ }}

\newcommand{\myendremark}{ $\blacksquare$ \bigskip}

\newcommand{\sfa}{{{\sf a}}}
\newcommand{\sfb}{{{\sf b}}}
\newcommand{\sfc}{{{\sf c}}}
\newcommand{\sfd}{{{\sf d}}}
\newcommand{\sfe}{{{\sf e}}}
\newcommand{\sff}{{{\sf f}}}
\newcommand{\sfg}{{{\sf g}}}
\newcommand{\sfh}{{{\sf h}}}
\newcommand{\sfi}{{{\sf i}}}
\newcommand{\bsfa}{{\mbox{\textsf{\textbf{a}}}}}
\newcommand{\bsfb}{{\mbox{\textsf{\textbf{b}}}}}
\newcommand{\bsfc}{{\mbox{\textsf{\textbf{c}}}}}
\newcommand{\bsfd}{{\mbox{\textsf{\textbf{d}}}}}
\newcommand{\bsfe}{{\mbox{\textsf{\textbf{e}}}}}
\newcommand{\bsff}{{\mbox{\textsf{\textbf{f}}}}}
\newcommand{\bsfg}{{\mbox{\textsf{\textbf{g}}}}}
\newcommand{\bsfh}{{\mbox{\textsf{\textbf{h}}}}}
\newcommand{\bsfi}{{\mbox{\textsf{\textbf{i}}}}}

\newcommand{\scra}{{\mathcal{A}}}
\newcommand{\scrb}{{\mathcal{B}}}
\newcommand{\scrc}{{\mathcal{C}}}
\newcommand{\bfscra}{{\bm{\mathcal{A}}}}
\newcommand{\bfscrb}{{\bm{\mathcal{B}}}}
\newcommand{\bfscrc}{{\bm{\mathcal{C}}}}
\newcommand{\bfscrap}{{\bm{\mathcal{A}'}}}
\newcommand{\bfscrbp}{{\bm{\mathcal{B}'}}}
\newcommand{\bfscrcp}{{\bm{\mathcal{C}'}}}
\newcommand{\bfscrapp}{{\bm{\mathcal{A}''}}}
\newcommand{\bfscrbpp}{{\bm{\mathcal{B}''}}}
\newcommand{\bfscrcpp}{{\bm{\mathcal{C}''}}}
\newcommand{\scrd}{{\mathcal{D}}}
\newcommand{\scre}{{\mathcal{E}}}
\newcommand{\scrf}{{\mathcal{F}}}
\newcommand{\scrg}{{\mathcal{G}}}
\newcommand{\scrgg}{{\mathscr{g}}}  
\newcommand{\scrh}{{\mathcal{H}}}
\newcommand{\scri}{{\mathcal{I}}}
\newcommand{\scrj}{{\mathcal{J}}}
\newcommand{\scrk}{{\mathcal{K}}}
\newcommand{\scrl}{{\mathcal{L}}}
\newcommand{\scrm}{{\mathcal{M}}}
\newcommand{\scrn}{{\mathcal{N}}}
\newcommand{\scro}{{\mathcal{O}}}
\newcommand\scroo{
  \mathchoice
    {{\scriptstyle\mathcal{O}}}
    {{\scriptstyle\mathcal{O}}}
    {{\scriptscriptstyle\mathcal{O}}}
    {\scalebox{0.6}{$\scriptscriptstyle\mathcal{O}$}}
  }
\newcommand{\scrp}{{\mathcal{P}}}
\newcommand{\scrq}{{\mathcal{Q}}}
\newcommand{\scrr}{{\mathcal{R}}}
\newcommand{\scrs}{{\mathcal{S}}}
\newcommand{\scrss}{{\mathscr{s}}}  
\newcommand{\scrt}{{\mathcal{T}}}
\newcommand{\scrv}{{\mathcal{V}}}
\newcommand{\scrw}{{\mathcal{W}}}
\newcommand{\scrz}{{\mathcal{Z}}}
\newcommand{\SP}{{\mathcal{SP}}}
\newcommand{\ST}{{\mathcal{ST}}}

\newcommand{\bzero}{{\bm{0}}}

\newcommand{\rbf}{v}

\newcommand{\dperm}{{\mathfrak{D}}}
\newcommand{\dcycle}{{\mathfrak{DC}}}
\newcommand{\GCdperm}{{}^{{\rm GC}}{\mathfrak{D}}}
\newcommand{\GCP}{{}^{{\rm GC}}P}

\newcommand{\restrict}{\upharpoonright}
\newcommand{\sinv}{\sigma^{-1}}
\def\pitilde{{\widetilde{\pi}}}
\def\omegahat{{\widehat{\omega}}}
\def\xihat{{\widehat{\xi}}}

\newcommand{\textbfit}[1]{\textbf{\textit{#1}}}

\newcommand{\be}{\begin{equation}}
\newcommand{\ee}{\end{equation}}

\newcommand{\fS}{\mathfrak{S}}
\newcommand{\Sym}{\mathfrak{S}}

\newcommand{\lev}{{\rm lev}}
\newcommand{\stat}{{\rm stat}}
\newcommand{\cyc}{{\rm cyc}}
\newcommand{\mysteryone}{{\rm mys1}}
\newcommand{\mysterytwo}{{\rm mys2}}
\newcommand{\Asc}{{\rm Asc}}
\newcommand{\asc}{{\rm asc}}
\newcommand{\Des}{{\rm Des}}
\newcommand{\des}{{\rm des}}
\newcommand{\Exc}{{\rm Exc}}

\newcommand{\EArec}{{\rm EArec}}
\newcommand{\earec}{{\rm earec}}
\newcommand{\recarec}{{\rm recarec}}
\newcommand{\erec}{{\rm erec}}
\newcommand{\nonrec}{{\rm nonrec}}
\newcommand{\nrar}{{\rm nrar}}
\newcommand{\ereccval}{{\rm ereccval}}
\newcommand{\ereccdrise}{{\rm ereccdrise}}
\newcommand{\eareccpeak}{{\rm eareccpeak}}
\newcommand{\eareccdfall}{{\rm eareccdfall}}
\newcommand{\eareccval}{{\rm eareccval}}
\newcommand{\ereccpeak}{{\rm ereccpeak}}
\newcommand{\rar}{{\rm rar}}
\newcommand{\evenrar}{{\rm evenrar}}
\newcommand{\oddrar}{{\rm oddrar}}
\newcommand{\nrcpeak}{{\rm nrcpeak}}
\newcommand{\nrcval}{{\rm nrcval}}
\newcommand{\nrcdrise}{{\rm nrcdrise}}
\newcommand{\nrcdfall}{{\rm nrcdfall}}
\newcommand{\nrfix}{{\rm nrfix}}
\newcommand{\Evenfix}{{\rm Evenfix}}
\newcommand{\evenfix}{{\rm evenfix}}
\newcommand{\Oddfix}{{\rm Oddfix}}
\newcommand{\oddfix}{{\rm oddfix}}
\newcommand{\evennrfix}{{\rm evennrfix}}
\newcommand{\oddnrfix}{{\rm oddnrfix}}
\newcommand{\Cpeak}{{\rm Cpeak}}
\newcommand{\cpeak}{{\rm cpeak}}
\newcommand{\Cval}{{\rm Cval}}
\newcommand{\cval}{{\rm cval}}
\newcommand{\Cdasc}{{\rm Cdasc}}
\newcommand{\cdasc}{{\rm cdasc}}
\newcommand{\Cddes}{{\rm Cddes}}
\newcommand{\cddes}{{\rm cddes}}
\newcommand{\Cdrise}{{\rm Cdrise}}
\newcommand{\cdrise}{{\rm cdrise}}
\newcommand{\Cdfall}{{\rm Cdfall}}
\newcommand{\cdfall}{{\rm cdfall}}

\newcommand{\maxpeak}{{\rm maxpeak}}
\newcommand{\nmaxpeak}{{\rm nmaxpeak}}
\newcommand{\minval}{{\rm minval}}
\newcommand{\nminval}{{\rm nminval}}

\newcommand{\exc}{{\rm exc}}
\newcommand{\excee}{{\rm excee}}
\newcommand{\exceo}{{\rm exceo}}
\newcommand{\excoe}{{\rm excoe}}
\newcommand{\excoo}{{\rm excoo}}
\newcommand{\erecexcoe}{{\rm erecexcoe}}
\newcommand{\erecexcoo}{{\rm erecexcoo}}
\newcommand{\nrexcoe}{{\rm nrexcoe}}
\newcommand{\nrexcoo}{{\rm nrexcoo}}
\newcommand{\aexc}{{\rm aexc}}
\newcommand{\aexcee}{{\rm aexcee}}
\newcommand{\aexceo}{{\rm aexceo}}
\newcommand{\aexcoe}{{\rm aexcoe}}
\newcommand{\aexcoo}{{\rm aexcoo}}
\newcommand{\earecaexcee}{{\rm earecaexcee}}
\newcommand{\earecaexceo}{{\rm earecaexceo}}
\newcommand{\nraexcee}{{\rm nraexcee}}
\newcommand{\nraexceo}{{\rm nraexceo}}
\newcommand{\Fix}{{\rm Fix}}
\newcommand{\fix}{{\rm fix}}
\newcommand{\fixe}{{\rm fixe}}
\newcommand{\fixo}{{\rm fixo}}
\newcommand{\rare}{{\rm rare}}
\newcommand{\raro}{{\rm raro}}
\newcommand{\nrfixe}{{\rm nrfixe}}
\newcommand{\nrfixo}{{\rm nrfixo}}
\newcommand{\xee}{x_{\rm ee}}
\newcommand{\xeo}{x_{\rm eo}}
\newcommand{\uee}{u_{\rm ee}}
\newcommand{\ueo}{u_{\rm eo}}
\newcommand{\yoo}{y_{\rm oo}}
\newcommand{\yoe}{y_{\rm oe}}
\newcommand{\voo}{v_{\rm oo}}
\newcommand{\voe}{v_{\rm oe}}
\newcommand{\zo}{z_{\rm o}}
\newcommand{\ze}{z_{\rm e}}
\newcommand{\wo}{w_{\rm o}}
\newcommand{\we}{w_{\rm e}}
\newcommand{\so}{s_{\rm o}}
\newcommand{\se}{s_{\rm e}}

\newcommand{\Wex}{{\rm Wex}}
\newcommand{\wex}{{\rm wex}}
\newcommand{\lrmax}{{\rm lrmax}}
\newcommand{\rlmax}{{\rm rlmax}}
\newcommand{\Rec}{{\rm Rec}}
\newcommand{\rec}{{\rm rec}}
\newcommand{\Arec}{{\rm Arec}}
\newcommand{\arec}{{\rm arec}}
\newcommand{\ERec}{{\rm ERec}}
\newcommand{\Val}{{\rm Val}}
\newcommand{\val}{{\rm val}}
\newcommand{\dasc}{{\rm dasc}}
\newcommand{\ddes}{{\rm ddes}}
\newcommand{\inv}{{\rm inv}}
\newcommand{\maj}{{\rm maj}}
\newcommand{\rs}{{\rm rs}}
\newcommand{\cross}{{\rm cr}}
\newcommand{\crosshat}{{\widehat{\rm cr}}}
\newcommand{\nest}{{\rm ne}}
\newcommand{\ucross}{{\rm ucross}}
\newcommand{\ucrosscval}{{\rm ucrosscval}}
\newcommand{\ucrosscpeak}{{\rm ucrosscpeak}}
\newcommand{\ucrosscdrise}{{\rm ucrosscdrise}}
\newcommand{\lcross}{{\rm lcross}}
\newcommand{\lcrosscpeak}{{\rm lcrosscpeak}}
\newcommand{\lcrosscval}{{\rm lcrosscval}}
\newcommand{\lcrosscdfall}{{\rm lcrosscdfall}}
\newcommand{\unest}{{\rm unest}}
\newcommand{\unestcval}{{\rm unestcval}}
\newcommand{\unestcpeak}{{\rm unestcpeak}}
\newcommand{\unestcdrise}{{\rm unestcdrise}}
\newcommand{\lnest}{{\rm lnest}}
\newcommand{\lnestcpeak}{{\rm lnestcpeak}}
\newcommand{\lnestcval}{{\rm lnestcval}}
\newcommand{\lnestcdfall}{{\rm lnestcdfall}}
\newcommand{\ujoin}{{\rm ujoin}}
\newcommand{\ljoin}{{\rm ljoin}}
\newcommand{\psnest}{{\rm psnest}}
\newcommand{\upsnest}{{\rm upsnest}}
\newcommand{\lpsnest}{{\rm lpsnest}}
\newcommand{\epsnest}{{\rm epsnest}}
\newcommand{\opsnest}{{\rm opsnest}}
\newcommand{\rodd}{{\rm rodd}}
\newcommand{\reven}{{\rm reven}}
\newcommand{\lodd}{{\rm lodd}}
\newcommand{\leven}{{\rm leven}}
\newcommand{\sg}{{\rm sg}}
\newcommand{\bl}{{\rm bl}}
\newcommand{\tran}{{\rm tr}}
\newcommand{\area}{{\rm area}}
\newcommand{\ret}{{\rm ret}}
\newcommand{\peaks}{{\rm peaks}}
\newcommand{\hl}{{\rm hl}}
\newcommand{\sll}{{\rm sl}}
\newcommand{\negg}{{\rm neg}}
\newcommand{\imp}{{\rm imp}}
\newcommand{\osg}{{\rm osg}}
\newcommand{\ons}{{\rm ons}}
\newcommand{\isg}{{\rm isg}}
\newcommand{\ins}{{\rm ins}}
\newcommand{\LL}{{\rm LL}}
\newcommand{\height}{{\rm ht}}
\newcommand{\as}{{\rm as}}

\newcommand{\ahat}{{\widehat{a}}}
\newcommand{\vhat}{{\widehat{v}}}
\newcommand{\yhat}{{\widehat{y}}}
\newcommand{\phat}{{\widehat{p}}}
\newcommand{\qhat}{{\widehat{q}}}
\newcommand{\Zhat}{{\widehat{Z}}}
\newcommand{\myhat}{{\widehat{\;}}}
\newcommand{\vtilde}{{\widetilde{v}}}
\newcommand{\ytilde}{{\widetilde{y}}}

\newcommand{\Stat}[1]{{\rm Stat}#1}

\newcommand{\erecop}{{\rm erecop}}
\newcommand{\erecin}{{\rm erecin}}
\newcommand{\nerecop}{{\rm nerecop}}
\newcommand{\nerecin}{{\rm nerecin}}
\newcommand{\crrtilde}{{\widetilde{{\rm cr}}}}
\newcommand{\crop}{{\rm crop}}
\newcommand{\crin}{{\rm crin}}
\newcommand{\nee}{{\rm ne}}
\newcommand{\neetilde}{{\widetilde{{\rm ne}}}}
\newcommand{\neop}{{\rm neop}}
\newcommand{\nein}{{\rm nein}}
\newcommand{\psne}{{\rm psne}}
\newcommand{\psnetilde}{{\widetilde{{\rm psne}}}}
\newcommand{\crne}{{\rm crne}}
\newcommand{\qne}{{\rm qne}}
\newcommand{\qnetilde}{{\widetilde{{\rm qne}}}}
\newcommand{\itilde}{{\widetilde{i}}}
\newcommand{\ktilde}{{\widetilde{k}}}
\newcommand{\ov}{{\rm ov}}
\newcommand{\ovin}{{\rm ovin}}
\newcommand{\ovinrev}{{\rm ovinrev}}
\newcommand{\ovtilde}{{\widetilde{{\rm ov}}}}
\newcommand{\ovbar}{{\overline{{\rm ov}}}}
\newcommand{\cov}{{\rm cov}}
\newcommand{\covin}{{\rm covin}}
\newcommand{\covinrev}{{\rm covinrev}}
\newcommand{\covtilde}{{\widetilde{{\rm cov}}}}
\newcommand{\covbar}{{\overline{{\rm cov}}}}
\newcommand{\qcov}{{\rm qcov}}
\newcommand{\qcovtilde}{{\widetilde{{\rm qcov}}}}
\newcommand{\pscov}{{\rm pscov}}
\newcommand{\brec}{{\rm brec}}
\newcommand{\brecop}{{\rm brecop}}
\newcommand{\brecin}{{\rm brecin}}
\newcommand{\nbrecop}{{\rm nbrecop}}
\newcommand{\nbrecin}{{\rm nbrecin}}

\newcommand{\op}{{\rm op}}
\newcommand{\inside}{{\rm in}}

\newcommand{\cros}{{\rm cr}}
\newcommand{\nes}{{\rm ne}}
\newcommand{\ccc}{{\rm cc}}
\newcommand{\ecp}{{\rm ecp}}
\newcommand{\ecpar}{{\rm ecpar}}
\newcommand{\ecpnar}{{\rm ecpnar}}
\newcommand{\ocp}{{\rm ocp}}
\newcommand{\ocpar}{{\rm ocpar}}
\newcommand{\ocpnar}{{\rm ocpnar}}
\newcommand{\ecvr}{{\rm ecvr}}
\newcommand{\ecvnr}{{\rm ecvnr}}
\newcommand{\ocvr}{{\rm ocvr}}
\newcommand{\ocvnr}{{\rm ocvnr}}
\newcommand{\ecr}{{\rm ecr}}
\newcommand{\ocr}{{\rm ocr}}
\newcommand{\ene}{{\rm ene}}
\newcommand{\oone}{{\rm one}}

\newcommand{\roof}{{\rm roof}}

\newcommand{\laguerre}[1]{\left. L\right|_{#1}}
\newcommand{\laguerrep}[1]{\left. L'\right|_{#1}}

\newcommand{\FZbij}{\varphi^{\rm FZ}}
\newcommand{\DSbij}{\varphi^{\rm DS}}
\newcommand{\SZbij}{\varphi^{\rm KZ}}
\newcommand{\SZbijperfect}{\varphi^{\rm KZ}}
\newcommand{\CDDSY}{\varphi^{\rm CDDSY}}

\newcommand{\VT}{{\mathcal{VT}}}
\newcommand{\OT}{{\mathcal{OT}}}

\newcommand{\motzkinperm}{\mathfrak{S}\mathcal{M}}
\newcommand{\schroderset}{\mathfrak{D}\mathcal{S}}
\newcommand{\motzkinset}{\Pi\mathcal{M}}
\newcommand{\dyckperfect}{\mathcal{ID}}

\newcommand{\xicomplement}{\xi^{\rm c}}

\newcommand{\GCsigma}{{\phi^{{\rm GC}}_{2n}(\sigma)}}

\newcommand{\pmp}{{\rm pmp}}

\newcommand{\circlesign}[1]{
\begin{tikzpicture}[baseline=(X.base), inner sep=1]
    \node[draw,circle] (X)  {\ensuremath{#1}};
\end{tikzpicture}
}

\newcommand{\circlesignsubscript}[1]{
\begin{tikzpicture}[baseline=(X.base), inner sep=0.5]
    \node[draw,circle] (X)  {\ensuremath{#1}};
\end{tikzpicture}
}

\clearpage

\tableofcontents

\clearpage

\section{Introduction}

The purpose of this paper is to exhibit several instances of the cyclic sieving phenomena 
(see Section~\ref{sec.intro.csp} for definitions)
using the combinatorics of continued fractions.
We will specifically study the following combinatorial families of objects:
permutations, set partitions, perfect matchings, D-permutations.

In a recent work, Adams, Elder, Lafreni\`ere, McNicholas, Striker and Welch
\cite{Adams_24},
carried out a systematic search for instances of the cyclic sieving phenomenon on permutations.
They searched the entire FindStat database \cite{findstat} for maps and statistics on permutations.
Their study involved three classes of involutions on permutations 
which they showed to be instances of the $q=-1$ phenomenon of Stembridge \cite{Stembridge_94}.
Along with several other contributions,
Adams et al initiated the study of the \textbf{{\em Corteel involution}} \cite[Map~239]{findstat},~\cite{Corteel_07}
in the context of the cyclic sieving phenomenon.
This involution uses the Foata--Zeilberger bijection \cite{Foata_90};
it exchanges the number of crossings and nestings in permutations 
and is constructed by taking ``complement'' of associated labelled Motzkin paths.
In this present paper, we continue this study that Adams et al initiated for the Corteel involution.
This paper attempts to convince the reader that
a natural setting to obtain instances of the cyclic sieving phenomenon with respect to the Corteel involution
is via 
the combinatorial theory of continued fractions for multivariate polynomials
counting permutation statistics.
We will study this in great detail in Section~\ref{sec.cf.permutations}.

Our first main result in this context is Proposition~\ref{prop.perm.csp}
where we provide a very general recipe to construct instances of the cyclic sieving phenomenon
via three multivariate polynomials~\eqref{eq.def.Qntilde},~\eqref{def.Qn},~\eqref{def.Qn.pq}.
Using this general recipe, we will reprove several, but not all, of the 
results from \cite[Section~4]{Adams_24}.
Our second contribution in this context is to prove \cite[Conjecture~4.24]{Adams_24};
this conjecture asks for the equidistribution of two different permutation statistics:

%
%
%
%
%
%
\begin{conjecture}[{\cite[Conjecture~4.24]{Adams_24}}]
The number of weak excedances that are also mid-points of a decreasing subsequence of
length~3 (Statistic~373 in FindStat)
is equidistributed with the cycle descent number (Statistic~317 in FindStat) and thus exhibits the
cyclic sieving phenomenon under involutions with $1$ point for $n=0$ and $2^{n-1}$ points for all $n\geq 1$.
In particular, this statistic exhibits the cyclic sieving phenomenon with respect to the Corteel involution.
\label{conj.equidist}
\end{conjecture}
\noindent We will prove this conjecture in Theorem~\ref{thm.equidist}.
We will need two ingredients for our proof, 
the first ingredient is to rewrite the two statistics in terms of simpler statistics;
we do this in Lemma~\ref{lem.wexx.cdes.relations}.
The second ingredient, which we also think is independently an important contribution,
is to reformulate a theorem of Deb \cite[Theorem~3.1]{Deb_23}
in terms of {\em cycle valley minima}.
This reformulation is in the same spirit as what was done in \cite[Section~4.3]{Deb_23}
and in \cite[Section~4.1.3]{Deb-Sokal_genocchi};
this has been done in Theorem~\ref{thm.Deb}. 
The original theorem of Deb 
was first conjectured by Sokal and Zeng \cite[Theorem~3.1]{Sokal-Zeng_22}.

Our third contribution for permutations is the proof of another conjecture of Adams et al
\cite[Conjecture~4.23]{Adams_24}. 
We do this in Section~\ref{sec.stat123}.
However, the general tools used in the rest of this paper do not work in this situation.
We instead modify a proof of Adams et al to prove this result.

The Corteel involution forms the prototypical involution in this work.
After our study involving permutations, 
we will look at other combinatorial objects 
and analogous maps.
For set partitions, 
studied in Section~\ref{sec.setpart},
we will use an involution due to Kasraoui and Zeng \cite{Kasraoui_06};
for perfect matchings,
studied in Section~\ref{sec.perfect},
we will study the restriction of this involution.
The \textbf{{\em Kasraoui--Zeng involution}} 
also exchanges crossings and nestings,
but on set partitions and perfect matchings,
and is described by taking complements of labelled Motzkin paths
(these are now the {\em Charlier diagrams} of Viennot \cite{Viennot_83}.)
For both of these objects, we will again provide a general recipe to 
construct instances of the cyclic sieving phenomenon,
see Propositions~\ref{prop.setpart.csp} and~\ref{prop.perfect.csp}.
We show that the number of fixed points of these involutions are 
the {\em odd Fibonacci numbers} $F_{2n-1}$ \cite[A001519]{OEIS} (Lemma~\ref{lem.csp.fixed.KZ})
and $1$ (Lemma~\ref{lem.csp.fixed.KZ.perfect}), respectively.

Another involution on set partitions and perfect matchings is the 
\textbf{{\em Chen--Deng--Du--Stanley--Yan (CDDSY) involution}} \cite{Chen_07}.
Chen et al construct a bijection between set partitions and
{\em vacillating tableaux},
which are certain walks on the Young's lattice.
They use this bijection to construct their involution.
We show that this involution also has the same number of fixed points as the 
Kasraoui--Zeng involution 
(Lemmas~\ref{lem.csp.fixed.CDDSY} and~\ref{lem.csp.fixed.CDDSY.perfect}).

We will then shift our attention to {\em D-permutations} (Section~\ref{sec.dperm});
these are a class of permutations of $[2n]$ and are counted by the median Genocchi numbers.
They were introduced by Lazar during his PhD \cite{Lazar_20,Lazar_22,Lazar_23}.
More recently,
Deb and Sokal constructed a bijection that resembles the Foata--Zeilberger bijection \cite[Section~6]{Deb-Sokal_genocchi}.
Hence it is natural to ask if there is an adaptation of the Corteel involution to this setting.
We will introduce such an involution in our paper (Section~\ref{sec.genocchi.corteel});
we will call this the \textbf{{\em Genocchi--Corteel involution}}.
Our construction is such that it can be restricted to the subclasses of 
{\em D-semiderangements} and also {\em D-derangements}.
We count the number of fixed points for D-permutations
and its subclasses of D-semiderangements, and D-derangements in Corollary~\ref{lem.genocchi.corteel.fixed.points};
the number of fixed points are given by 
$2^n F_{2n-1}$ \cite[A154626]{OEIS},
$3^{n-1}$ \cite[A133494]{OEIS}
and $2^{n-1}$ \cite[A011782]{OEIS}, respectively.
We will then use this involution on D-permutation and its subclasses to exhibit the cyclic sieving phenomena for various statistics;
we provide a general recipe for this in Propositions~\ref{prop.dperm.csp} and~\ref{prop.csp.dperm.subclass}.

%

We conclude this introduction with three remarks:
\begin{enumerate}

\item Prior to the work in \cite{Adams_24}, 
the method of searching the FindStat database for instances 
of dynamical phenomena was pioneered by
Elder, Lafreni\'ere, McNicholas, Striker, and Welch in
\cite{Elder_23}.
In this article, the authors searched the FindStat database for pairs of maps and statistics 
which were instances of the {\em homomesy phenomenon}.

\item 
At this moment, the FindStat database contains statistics on set partitions and perfect matchings. 
It will be interesting to perform a systematic search on these classes 
for the cyclic sieving phenomenon
similar to the search performed by Adams et al in \cite{Adams_24}.

\item An important ingredient of our results is plugging in $-1$ 
in some known continued fraction identity to yield a simple rational function.
We mention a recent work \cite{Deb-Sokal_cycle_minus_one}
where results were obtained by 
plugging in $-1$ to a statistic, 
viz. to the number of cycles.
\end{enumerate}

The rest of this introduction is structured as follows:
in Section~\ref{sec.intro.csp}, we define the cyclic sieving phenomenon.
Then in Section~\ref{sec.intro.outline}, we provide the outline of the rest of this paper.

\subsection{The cyclic sieving phenomena}
\label{sec.intro.csp}

The cyclic sieving phenomenon was introduced by Reiner, Stanton and White \cite{Reiner_04};
also see the nice survey article of Sagan \cite{Sagan_11} and also \cite{Reiner_14}.

Let $X$ be a finite set (usually a set of combinatorial objects).
Next, let $C$ be a finite cyclic group of order $n$
with an action on $X$.
Given $g\in C$, we use $X^g$ to denote the elements of $X$ invariant under the action of $g$,
i.e.
\be
X^g 
\;=\; 
\{y\in X \;\colon\; g\cdot y =y\}\;.
\ee
Let $\omega$ be a generator of the group $C$, i.e., $C= \langle \omega \rangle$.
We are now ready to define the cyclic sieving phenomenon.

\begin{definition}
Given a set $X$, a polynomial $f(q)$, and an action of $C$ on $X$,
the triple $(X, C, f(q))$ exhibits the {\em cyclic sieving phenomenon} if,
for all $d$ we have
\be
	\left|X^{\omega^d}\right| \;=\; f\left(\exp\left(2\pi i \dfrac{d}{n}\right)\right) \;.
\label{eq.def.csp}
\ee
\label{def.csp}
\end{definition}

Notice that when $d=n$, the identity~\eqref{eq.def.csp} implies that $f(1) = |X|$.
Thus, the coefficients of the polynomial $f(q)$ should sum up to the cardinality of $X$.
For our purposes, $f(q)$ will be the generating polynomial of some {\em statistic on the set $X$}.

When $n=2$, the action of the generator $\omega$ 
gives a bijection on the set $X$, $x \mapsto \omega\cdot x$.
This bijection is, in fact, an involution.
This special case, was first introduced as the ``$q=-1$ phenomenon'' by Stembridge \cite{Stembridge_94}.
Exhibiting the cyclic sieving phenomenon in this case is equivalent 
to showing that $f(1) = |X|$,
and that $f(-1) = |X^\omega|$, the number of fixed points of the involution $x\mapsto \omega\cdot x$.
In this paper, our instances of the cyclic sieving phenomenon
will all be instances of the ``$q=-1$'' phenomenon.

The main idea in this paper can be summarised as follows:
Let $(S_n)_{n\geq 0}$ be a combinatorial family of objects
and let $\Stat: S_n\to \mathbb{N}$ be some statistics on this set.
Let $F(t;q) = \sum_{n=0}^\infty \sum_{\sigma\in S_n} q^{\Stat{(\sigma})} t^n$ 
be its bivariate ordinary generating function.
Here the indeterminate $q$ will keep track of this statistic and $t$ will keep track of the set $S_n$.
Also, let $\phi_n :S_n \to S_n$ be an involution such that the number fixed points of this involution is $a_n$.
We think of this involution as the action of $\omega$ on the set $S_n$ where $\omega^2 = {\rm id}_C$.
Then exhibiting the cyclic sieving phenomenon is equivalent to showing 
$F(t;-1) = \sum_{n=0}^\infty a_n t^n$.
We will specifically look at generating functions $F(t;q)$ which can be expressed as a nice classical continued fraction;
this will simplify the evaluation of the substitution $F(t;-1)$
and our results will seem to be easy consequences.
Following \cite{Adams_24}, we will say that the statistic
$\Stat$ exhibits the cyclic sieving phenomenon to mean that
the corresponding generating polynomial
$\sum_{\sigma\in S_n} q^{\Stat{(\sigma})}$
exhibits the cyclic sieving phenomenon.

\subsection{Outline of the paper}
\label{sec.intro.outline}

The rest of the paper is structured as follows:
in Section~\ref{sec.labelled.paths} we will 
establish notation for labelled 
Dyck, Motzkin and Schr\"oder paths.
We will then present in Sections~\ref{sec.cf.permutations}--\ref{sec.dperm} 
our results for permutations, set partitions, perfect matchings and finally, D-permutations and its subclasses.

In Appendix~\ref{app.positivity.vincular},
we present some positivity conjectures on
the enumeration of certain vincular patterns in permutations
which we discovered while writing this paper.


\section{Labelled Dyck, Motzkin and Schr\"oder paths}
\label{sec.labelled.paths}

Recall that a \textbfit{Motzkin path} of length $n \ge 0$
is a path $\omega = (\omega_0,\ldots,\omega_n)$
in the right quadrant $\N \times \N$,
starting at $\omega_0 = (0,0)$ and ending at $\omega_n = (n,0)$,
whose steps $s_j = \omega_j - \omega_{j-1}$
are $(1,1)$ [``rise'' or ``up step''], $(1,-1)$ [``fall'' or ``down step'']
or $(1,0)$ [``level step''].
We write $h_j$ for the \textbfit{height} of the Motzkin path at abscissa~$j$,
i.e.\ $\omega_j = (j,h_j)$;
note in particular that $h_0 = h_n = 0$.
We write $\scrm_n$ for the set of Motzkin paths of length~$n$,
and $\scrm = \bigcup_{n=0}^\infty \scrm_n$.
A Motzkin path is called a \textbfit{Dyck path} if it has no level steps.
A Dyck path always has even length;
we write $\scrd_{2n}$ for the set of Dyck paths of length~$2n$,
and $\scrd = \bigcup_{n=0}^\infty \scrd_{2n}$.

A \textbfit{Schr\"oder path} of length $2n$ ($n \ge 0$)
is a path $\omega = (\omega_0,\ldots,\omega_{2n})$
in the right quadrant $\N \times \N$,
starting at $\omega_0 = (0,0)$ and ending at $\omega_{2n} = (2n,0)$,
whose steps are $(1,1)$ [``rise'' or ``up step''],
$(1,-1)$ [``fall'' or ``down step'']
or $(2,0)$ [``long level step''].
We write $s_j$ for the step starting at abscissa $j-1$.
If the step $s_j$ is a rise or a fall,
we set $s_j = \omega_j - \omega_{j-1}$ as before.
If the step $s_j$ is a long level step,
we set $s_j = \omega_{j+1} - \omega_{j-1}$ and leave $\omega_j$ undefined;
furthermore, in this case there is no step $s_{j+1}$.
We write $h_j$ for the height of the Schr\"oder path at abscissa~$j$
whenever this is defined, i.e.\ $\omega_j = (j,h_j)$.
Please note that $\omega_{2n} = (2n,0)$ and $h_{2n} = 0$
are always well-defined,
because there cannot be a long level step starting at abscissa $2n-1$.
Note also that a long level step at even (resp.~odd) height
can occur only at an odd-numbered (resp.~even-numbered) step.
We write $\scrs_{2n}$ for the set of Schr\"oder paths of length~$2n$,
and $\scrs = \bigcup_{n=0}^\infty \scrs_{2n}$.

There is an obvious bijection between Schr\"oder paths and Motzkin paths:
namely, every long level step is mapped onto a level step.

Let $\bfscra = (\scra_h)_{h \ge 0}$, $\bfscrb = (\scrb_h)_{h \ge 1}$
and $\bfscrc = (\scrc_h)_{h \ge 0}$ be sequences of finite sets.
An
\textbfit{$(\bfscra,\bfscrb,\bfscrc)$-labelled Motzkin path of length $\bm{n}$}
is a pair $(\omega,\xi)$
where $\omega = (\omega_0,\ldots,\omega_n)$
is a Motzkin path of length $n$,
and $\xi = (\xi_1,\ldots,\xi_n)$ is a sequence satisfying
\be
   \xi_i  \:\in\:
   \begin{cases}
       \scra(h_{i-1})  & \textrm{if step $i$ is a rise (i.e.\ $h_i = h_{i-1} + 1$)}
              \\[1mm]
       \scrb(h_{i-1})  & \textrm{if step $i$ is a fall (i.e.\ $h_i = h_{i-1} - 1$)}
              \\[1mm]
       \scrc(h_{i-1})  & \textrm{if step $i$ is a level step (i.e.\ $h_i = h_{i-1}$)}
   \end{cases}
 \label{eq.xi.ineq}
\ee
where $h_{i-1}$ (resp.~$h_i$) is the height of the Motzkin path
before (resp.~after) step $i$.
[For typographical clarity
 we have here written $\scra(h)$ as a synonym for $\scra_h$, etc.]
We call $\xi_i$ the \textbfit{label} associated to step $i$.
We call the pair $(\omega,\xi)$
an \textbfit{$(\bfscra,\bfscrb)$-labelled Dyck path}
if $\omega$ is a Dyck path (in this case $\bfscrc$ plays no role).
We denote by $\scrm_n(\bfscra,\bfscrb,\bfscrc)$
the set of $(\bfscra,\bfscrb,\bfscrc)$-labelled Motzkin paths of length $n$,
and by $\scrd_{2n}(\bfscra,\bfscrb)$
the set of $(\bfscra,\bfscrb)$-labelled Dyck paths of length $2n$.

We define a \textbfit{$(\bfscra,\bfscrb,\bfscrc)$-labelled
Schr\"oder path}
in an analogous way;
now the sets $\scrc_h$ refer to long level steps.
We denote by $\scrs_{2n}(\bfscra,\bfscrb,\bfscrc)$
the set of $(\bfscra,\bfscrb,\bfscrc)$-labelled Schr\"oder paths
of length $2n$.

Let us stress that the sets $\scra_h$, $\scrb_h$ and $\scrc_h$ are allowed
to be empty.
Whenever this happens, the path $\omega$ is forbidden to take a step
of the specified kind starting at the specified height.


\section{Permutations}
\label{sec.cf.permutations}

It is well known that the generating function 
of factorials can be represented as Stieltjes-type and Jacobi-type continued fractions as follows:
\begin{eqnarray}
\sum_{n=0}^\infty n!t^n
\;=\;
\cfrac{1}{1-\cfrac{t}{1-\cfrac{t}{1-\cfrac{2t}{1-\cfrac{2t}{1-\ldots}}}}}
\;=\;
\cfrac{1}{1- t - \cfrac{t}{1- 3 t - \cfrac{4 t}{1 - 5 t - \cfrac{9t}{1-\ldots}}}}
\;.
\end{eqnarray}
We will use various generalisations of this identity in this section.
We will replace $n!$ with various polynomials enumerating statistics on permutations.

In Section~\ref{subsec.statistics}, we introduce notation for permutation statistics to be used in the rest of this paper.
Then in Section~\ref{subsec.reformulation}, 
we reformulate a result of Deb  \cite[Theorem~3.1]{Deb_23};
this will be one of our main technical tools to prove Conjecture~\ref{conj.equidist}.
After this, we recall the Corteel involution in Section~\ref{subsec.Corteel.inv}
in terms of weighted paths.
Then in Section~\ref{subsec.szpoly.permutations},
we recall various multivariate polynomials studied by Sokal and Zeng in~\cite{Sokal-Zeng_22}.
With this background, 
in Section~\ref{subsec.csp.permutations},
we then state a recipe for producing instances of 
the cyclic sieving phenomenon on permutations involving the Corteel involution
and the Sokal--Zeng polynomials;
here we reprove some of the results of Adams et al, viz., 
\cite[Theorems~4.7,4.15]{Adams_24}
and the first statement in \cite[Corollary~4.8]{Adams_24}.
Then in Section~\ref{subsec.vincular.permutations},
we exhibit the cyclic sieving phenomenon for the count of some vincular patterns;
these statistics are not covered by the  Sokal--Zeng polynomials.
Here, we will reprove the second and third statements in 
\cite[Corollary~4.8]{Adams_24}
but we are unable to prove \cite[Theorem~4.20]{Adams_24}.

In Section~\ref{sec.Adams.conj},
we will prove Conjecture~\ref{conj.equidist} and reprove \cite[Theorem~4.10]{Adams_24}
using our continued-fractions set up.
Finally, in Section~\ref{sec.stat123},
we will prove \cite[Conjecture~4.23]{Adams_24}, respectively.
However, for this result, we rely on another technique.
In particular, we will use the notion of {\em positional marked patterns}
and an involution defined by Adams et al.

\subsection{Permutation statistics}
\label{subsec.statistics}

We will use the glossary of permutation statistics 
as stated in \cite[Section~2.5]{Deb_thesis};
this includes the record-and-cycle classification 
\cite[Section~2.5.1]{Deb_thesis},
and crossings~and~nestings
\cite[Section~2.5.2]{Deb_thesis}.

We will now describe a few additional statistics which will be used in this section.
Given a permutation $\sigma\in \mathfrak{S}_n$, 
we use $\wexx(\sigma)$
to denote \cite[Statistic~373]{findstat}:
it is {\em the number of weak excedances that are also mid-points of a decreasing subsequence of length 3}, i.e.,
\begin{eqnarray}
	\wexx(\sigma) &=& |\{j\:\colon  j\leq \sigma(j)\} \,\cap\,  \{j\:\colon \hbox{$\exists$ $i,k$ with $i<j<k$ and } \sigma(i)>\sigma(j)>\sigma(k)  \}| , 
	\nonumber\\
	\label{eq.def.wexx}
\end{eqnarray}
Also let $\cdes(\sigma)$ denote {\em the cycle descent number} given in
\cite[Statistic~317]{findstat}, i.e.,
\be
	\cdes(\sigma) \;=\; |\{i\:\colon i>\sigma(i) \hbox{ and } \sigma(i) \hbox{ is not the smallest element in its cycle} \}|\;.
	\label{eq.def.cdes}
\ee

\subsection{Reformulation of an identity of Deb using cycle valley minima}
\label{subsec.reformulation}

Few years ago, Sokal and Zeng \cite{Sokal-Zeng_22}
carried out a detailed investigation of continued fractions 
enumerating permutations, set partitions and perfect matchings with respect to a large 
(sometimes infinite)
 defined in~number of independent statistics.
They conjectured a continued fraction for permutations \cite[Conjecture~2.3]{Sokal-Zeng_22},
which was recently proved by Deb \cite[Theorem~3.1]{Deb_23}.
However, \cite[Theorem~3.1]{Deb_23}
does not explicitly include the count of the statistic {\em cycle valley minima} ($\minval$).

The main purpose of this section is to state a reformulation of this result
that includes the count of cycle valley minima;
this reformulation is in the same spirit as what was done by Deb and Sokal in~\cite[Section~4.1.3]{Deb-Sokal_genocchi} 
and by Deb in~\cite[Section~4.3]{Deb_23}.
Our result will help us to prove \cite[Theorem~4.10 and Conjecture~4.24]{Adams_24} in Section~\ref{sec.Adams.conj}.

For a permutation $\sigma\in \mathfrak{S}_n$, 
a non-singleton cycle contains precisely one maximum element,
which is necessarily a cycle peak,
and precisely one minimum element,
which is necessarily a cycle valley.
We now introduce a polynomial that is similar to
Sokal and Zeng's first master polynomial for permutations
\cite[eq.~(2.23)]{Sokal-Zeng_22},
except the classification of cycle valleys as records or non-records
is replaced by the classification of cycle valleys as
minima or non-minima:
\begin{eqnarray}
        & &
        \widetilde{Q}_n(x_1,x_2, \ytilde_1, y_2,u_1, u_2, \vtilde_1, v_2, \mathbf{w})
        \;=\;
        \nonumber\\[4mm]
         & &\qquad\qquad 
         \sum_{\sigma \in \mathfrak{S}_n}
        x_1^{\eareccpeak(\sigma)} x_2^{\eareccdfall(\sigma)}
        \ytilde_1^{\minval(\sigma)} y_2^{\ereccdrise(\sigma)}
        \:\times
        \qquad\qquad
        \nonumber\\[-1mm]
   & & \qquad\qquad\qquad\:
   u_1^{\nrcpeak(\sigma)} u_2^{\nrcdfall(\sigma)}
   \vtilde_1^{\nminval(\sigma)} v_2^{\nrcdrise(\sigma)}
        \mathbf{w}^{\boldsymbol{\fix}(\sigma)}
 \label{eq.def.Qntilde}
\end{eqnarray}
where
\begin{equation}
	\mathbf{w}^{\boldsymbol{\fix}(\sigma)} \;=\; \prod_{i\in \Fix} w_{\psnest(i,\sigma)}\;.
	\label{eq.wfix}
\end{equation}

We have the following theorem:
\begin{theorem}[{Reformulation of \cite[Theorem~3.1]{Deb_23} using cycle valley minima}]
The ordinary generating function of the polynomials $\widetilde{Q}_n$
defined in~\eqref{eq.def.Qntilde}
has a Jacobi-type continued fraction 
\begin{eqnarray}
   & & \hspace*{-7mm}
   \sum_{n=0}^\infty \widetilde{Q}_n(x_1,x_2, \ytilde_1, y_2,u_1, u_2, \vtilde_1, v_2, \mathbf{w}) t^n
   \;=\;
       \nonumber \\
   & & 
   \cfrac{1}{1 - w_0 t - \cfrac{x_1 \ytilde_1 t^2}{1 -  (x_2\!+\!y_2\!+\!w_1) t - \cfrac{(x_1\!+\!u_1)(\ytilde_1\!+\!\vtilde_1) t^2}{1 - (x_2\!+\!u_2\!+\!y_2\!+\!v_2\!+\!w_2)t - \cfrac{(x_1\!+\!2u_1)(\ytilde_1\!+\!2\vtilde_1) t^2}{1 - \cdots}}}}
       \nonumber \\[1mm]
   \label{eq.thm.perm.reform}
\end{eqnarray}

with coefficients
\begin{subeqnarray}
    \gamma_{0} & = & w_0
         \\[1mm]
    \gamma_{n}   & = & [x_2 \!+\! (n-1)u_2 ] + [y_2 \!+\! (n-1)v_2 ] +  w_n   \qquad\hbox{for $n \ge 1$}
         \\[1mm]
	\beta_n  & = &   [x_1 \!+\! (n-1)u_1 ] [\ytilde_1 \!+\! (n-1)\vtilde_1]
\label{def.weights.thm.Deb}
 \end{subeqnarray}
\label{thm.Deb}
\end{theorem}

\begin{proof} The proof follows from the proof of
	\cite[Theorem~3.1]{Deb_23}
	the only difference being that 
	we replace the factor $(\lambda + n-1)y_1$
	in \cite[eq.~(3.2c)]{Deb_23}
	with the factor 
	$\ytilde_1 + (n-1)\vtilde_1$.
	We can do this because the variable $\ytilde_1$ now keeps track 
	of cycle valley minima, 
	(which were shown to be cycle closers in \cite[Lemma~6.6]{Deb_23})
	and $\vtilde_1$ keep track of the other cycle valleys.
\end{proof}

\begin{remark}
Note that the factors $\lambda$ multiplying the variables $w_n$ corresponding to fixed points
in \cite[Theorem~3.1]{Deb_23}
have now disappeared because we are no longer giving such factors to singleton cycles.
Even more strikingly, our reformulation in Theorem~\ref{thm.Deb} now looks identical to 
Sokal and Zeng's \cite[Theorem~2.2]{Sokal-Zeng_22} 
except that $y_1,v_1$ have been replaced by $\ytilde_1,\vtilde_1$;
we will soon state this in Theorem~\ref{thm.Deb.prime}${{}^{ \prime}}$.
Moreover, Deb's second and third versions of continued fractions for permutation \cite[Theorems~3.2 and~3.3]{Deb_23}, respectively,
can also be rewritten in a similar way,
but we refrain from writing out the details.
\end{remark}

\subsection{The Corteel involution on permutations}
\label{subsec.Corteel.inv}

In \cite[Proposition~4]{Corteel_07}, Corteel showed that {\em the number of permutations with $l$
crossings and $m$ nestings is equal to the number of permutations with $m$ crossings and $l$ nestings.}\footnote{{\bf Warning:}
We should warn the reader that Corteel's definitions of crossings and nesting is different from ours; hers can be expressed in our notation as follows
\begin{subeqnarray}
{\rm crossings} 
\;=\;
{\ucross}
+\lcross + \cdrise \\
{\rm nestings} 
\;=\;
\unest +\lnest + \psnest
\label{eq.cross.nest.Corteel}
\end{subeqnarray}
See the historical remarks in \cite[Section~2.5]{Sokal-Zeng_22} for more on this story.
}
Her proof implicitly used an involution on permutations
but this was only explicitly written down in
\cite[Map~239]{findstat}. 
%
%

Corteel used the Foata--Zeilberger \cite{Foata_90} bijection
to construct her involution on permutations.
This bijection is a bijection
between permutations and labelled Motzkin paths
that has been very successfully
employed to obtain continued fractions involving polynomial coefficients
counting various permutation statistics (see for example
\cite{Randrianarivony_98b, Corteel_07, Blitvic_21, Sokal-Zeng_22}).
More recently, Deb provided a reinterpretation of this bijection
in terms of Laguerre digraphs \cite{Deb_23}; we will follow his presentation.

The Foata--Zeilberger bijection (as stated in \cite[Section~6.1]{Deb_23})
is a bijective correspondence between $\mathfrak{S}_n$
and the set of
$(\bfscra, \bfscrb,\bfscrc)$-labelled Motzkin paths of length $n$, 
where
the labels $\xi_i$ lie in the sets
\begin{subeqnarray}
	\scra_h        & = & \{0,\ldots, h\}     \,\;\;\qquad\qquad\qquad\quad \qquad \qquad\qquad\qquad \qquad\hbox{for $h \ge 0$}  \\
	\scrb_h        & = &  \{0,\ldots,h-1\}   \qquad\qquad\qquad\quad  \qquad\qquad\qquad \qquad\quad\hbox{for $h \ge 1$}  \\
	\scrc_h  & = &  
	\left(\{1\}\times C_h^{(1)}\right)
	\,\cup\, 
	\left(\{2\}\times C_h^{(2)}\right)
	\,\cup\,
	\left(\{3\}\times C_h^{(3)}\right)
	\qquad\hbox{for $h \ge 0$}  
 \label{def.abc.perm}
\end{subeqnarray}
where
\begin{subeqnarray}
	C_h^{(1)} & = & \{0,\ldots,h-1\} \qquad\hbox{for $h \ge 1$} \\
	C_h^{(2)} & = & \{0,\ldots,h-1\} \qquad\hbox{for $h \ge 1$}  \\
	C_h^{(3)} & = & \{0\}  \qquad\qquad\qquad\;\:\hbox{for $h \ge 0$} 
\label{def.c.perm}
\end{subeqnarray}

A level step that has label 
$\xi_h\in\{i\}\times C_h^{(i)}$
will be called a level step of type $i$ ($i=1,2,3$).
Let $\motzkinperm_{n}$ denote this set of $(\bfscra, \bfscrb,\bfscrc)$-labelled
Motzkin paths of length $n$ with label sets
given by~\eqref{def.abc.perm},\eqref{def.c.perm}.
The Foata--Zeilberger bijection as per \cite[Section~6.1]{Deb_23}
is a bijective map $\FZbij_n : \mathfrak{S}_n\to \motzkinperm_n$.
For the convenience of the reader, 
we relate our terminology to that of Corteel's \cite[Section~3.1]{Corteel_07}:
\begin{itemize}
\item rises
correspond to steps $N$,

\item falls correspond to steps $S$,

\item level steps of type $1$ correspond to steps $\bar{E}$,

\item Corteel considers cycle double rises and fixed points together, 
and hence, her steps $E$ correspond to both level steps of types $2$ and $3$.
\end{itemize}

For what follows, 
we denote by $c_0, c_1, c_2, \ldots, c_h$
the following labels at
a level step at height $h$:
\be
c_i 
\;=\;
\begin{cases}
(3,0) &\quad \text{if $ i = h$}\\
(2,i) &\quad \text{otherwise}
\end{cases}
\;.
\ee
Thus, $\scrc_h = \bigl(\{1\}\times C_{h}^{(1)}\bigr)\cup \{c_0, \ldots, c_h\}$.

We will now describe an involution $\phi_{n}: \motzkinperm_{n} \to \motzkinperm_{n}$;
the Corteel involution will then be obtained by the composition of bijections 
$\bigl(\FZbij_{n}\bigr)^{-1} \circ \phi_{n} \circ \FZbij_{n}$.

Let $(\omega,\xi)\in \motzkinperm_{n}$
where $\xi = (\xi_1,\ldots,\xi_{n})$.
Let $\xicomplement = (\xicomplement_1,\ldots, \xicomplement_{n})$ be a sequence where
$\xicomplement_i$
is given by
\be
\xicomplement_i
\;=\;
\begin{cases}
    h -\xi_i &\qquad \text{if step $i$ is a rise}\\
    h-1 - \xi_i & \qquad \text{if step $i$ is a fall}\\
	(1,h-1-k) &\qquad \text{if step $i$ is a level step of type 1
	and $\xi_i = (1,k)$}\\
	c_{h-k} &\qquad \text{if step $i$ is a level step of types $2$ or $3$ and $\xi_i = c_k$}
\end{cases}
\ee
It is clearly follows that $(\omega,\xicomplement)\in \motzkinperm_{n}$
and that $(\xicomplement)^{{\rm c}} = \xi$.
Thus, the map $(\omega, \xi)\mapsto (\omega,\xicomplement)$
is an involution on the set $\motzkinperm_{n}$;
we use $\phi_{n}$ to denote this involution.
The Corteel involution on permutations
is the the map $\phi_n^{{\rm C}}\coloneqq \bigl(\FZbij_{n}\bigr)^{-1} \circ \phi_{n} \circ \FZbij_{n}$.

An important property of the Corteel involution is that crossings and nestings, in the sense of Corteel, are exchanged.
More precisely, we state this in the following proposition without proof.

\begin{proposition}
For a permutation $\sigma\in\mathfrak{S}_n$ define $\widehat{\sigma} \coloneqq \phi_n^{{\rm C}}(\sigma)$.
Then if an index $i\in [n]$ is a cycle peak, cycle valley, or a cycle double fall in $\sigma$,
it continues to have the same cycle type in $\widehat{\sigma}$;
if $i\in \Cdrise(\sigma) \cup \Fix(\sigma)$  then 
$i\in \Cdrise(\widehat{\sigma}) \cup \Fix(\widehat{\sigma})$ .
Also, we have the following equidistribution
\begin{eqnarray}
&&
\hspace*{-15mm}
\bigl(\ucross(\sigma) + \lcross(\sigma) + \cdrise(\sigma), \unest(\sigma)+ \lnest(\sigma) +\psnest(\sigma)\bigr)
\nonumber\\
&=&
\bigl(\unest(\widehat{\sigma})+\lnest(\widehat{\sigma})+\psnest(\widehat{\sigma}),  \ucross(\widehat{\sigma})+ \lcross(\widehat{\sigma})+\cdrise(\widehat{\sigma}) \bigr) \;.
\end{eqnarray}
\end{proposition}
%
%

Adams et al counted the number of fixed points of the Corteel involution which we now recall.

\begin{lemma}[{\cite[Lemma~4]{Adams_24}}]
The number of fixed points of $\mathfrak{S}_n$
under the action of the Corteel involution $\phi_n^{\rm C}$
is $1$ when $n=0$ and $2^{n-1}$ when $n\geq 1$.
\end{lemma}

\noindent Additionally, we mention that that this sequence has the following 
Jacobi-type continued fraction, albeit of finite depth:
\be
1 + \sum_{n=1}^\infty 2^{n-1} t^n
\;=\;
\cfrac{1}{1 -  t - \cfrac{t^2}{1 - t}}\;.
\label{eq.corteel.fp.contfrac}
\ee

\subsection{Sokal and Zeng's
multivariate polynomials enumerating cyclic statistics on permutations}
\label{subsec.szpoly.permutations}

In this section, we recall two of Sokal and Zeng's multivariate polynomials enumerating
statistics on  permutations;
first we shall state the first polynomial \cite[eq.~(2.23)]{Sokal-Zeng_22},
then we shall state their first $p,q$-generalisation involving crossings and nestings \cite[eq.~(2.51)]{Sokal-Zeng_22}.
We recall that
we will use the glossary of permutation statistics 
as stated in \cite[Section~2.5]{Deb_thesis};
this includes the record-and-cycle classification 
\cite[Section~2.5.1]{Deb_thesis},
and crossings~and~nestings
\cite[Section~2.5.2]{Deb_thesis}.

We first recall the polynomial \cite[eq.~(2.23)]{Sokal-Zeng_22}:
\begin{eqnarray}
   & &
   Q_n^{(1)}(x_1,x_2,y_1,y_2,u_1,u_2,v_1,v_2,\mathbf{w})
   \;=\;
       \nonumber \\[4mm]
   & & \qquad\qquad
   \sum_{\sigma \in \mathfrak{S}_{n}}
   x_1^{\eareccpeak(\sigma)} x_2^{\eareccdfall(\sigma)}
   y_1^{\ereccval(\sigma)} y_2^{\ereccdrise(\sigma)}
   \:\times
       \qquad\qquad
       \nonumber \\[-1mm]
   & & \qquad\qquad\qquad\:
   u_1^{\nrcpeak(\sigma)} u_2^{\nrcdfall(\sigma)}
   v_1^{\nrcval(\sigma)} v_2^{\nrcdrise(\sigma)}
   \mathbf{w}^{\boldsymbol{\fix}(\sigma)}
 \label{def.Qn}
\end{eqnarray}
where $\mathbf{w}^{\boldsymbol{\fix}(\sigma)}$ was defined in \eqref{eq.wfix}.
The polynomials \reff{def.Qn} have a beautiful J-fraction 
\cite[Theorem~2.2]{Sokal-Zeng_22}.

Next we recall the polynomial \cite[eq.~(2.51)]{Sokal-Zeng_22}
which is a $p,q$-generalisation of the polynomials $Q_n^{(1)}$:
\begin{eqnarray}
   & &
   \hspace*{-14mm}
   Q_n^{(2)}(x_1,x_2,y_1,y_2,u_1,u_2,v_1,v_2,\mathbf{w},
   p_{+1},p_{+2},p_{-1},p_{-2},q_{+1},q_{+2},q_{-1},q_{-2},s)
   \;=\;
   \hspace*{-1cm}
       \nonumber \\[4mm]
   & & \qquad\qquad
   \sum_{\sigma \in \mathfrak{S}_{n}}
   x_1^{\eareccpeak(\sigma)} x_2^{\eareccdfall(\sigma)}
   y_1^{\ereccval(\sigma)} y_2^{\ereccdrise(\sigma)}
   \:\times
       \qquad\qquad
       \nonumber \\[-1mm]
   & & \qquad\qquad\qquad\:
   u_1^{\nrcpeak(\sigma)} u_2^{\nrcdfall(\sigma)}
   v_1^{\nrcval(\sigma)} v_2^{\nrcdrise(\sigma)}
   \mathbf{w}^{\boldsymbol{\fix}(\sigma)}
   \:\times
       \qquad\qquad
       \nonumber \\[3mm]
   & & \qquad\qquad\qquad\:
   p_{+1}^{\ucrosscval(\sigma)}
   p_{+2}^{\ucrosscdrise(\sigma)}
   p_{-1}^{\lcrosscpeak(\sigma)}
   p_{-2}^{\lcrosscdfall(\sigma)}
          \:\times
       \qquad\qquad
       \nonumber \\[3mm]
   & & \qquad\qquad\qquad\:
   q_{+1}^{\unestcval(\sigma)}
   q_{+2}^{\unestcdrise(\sigma)}
   q_{-1}^{\lnestcpeak(\sigma)}
   q_{-2}^{\lnestcdfall(\sigma)}
   s^{\psnest(\sigma)}\;.
 \label{def.Qn.pq}
\end{eqnarray}
These polynomials~\eqref{def.Qn.pq} also have a nice continued fraction
\cite[Theorem~2.7]{Sokal-Zeng_22}.
Sokal and Zeng have also introduced several other multivariate polynomials in their paper,
in particular, their {\em second} polynomials involve the count of cycles. 
We, however, refrain from working with those.
We will instead use our polynomial \eqref{eq.def.Qntilde} 
which has the continued fraction in Theorem~\ref{thm.Deb}
which generalises their second J-fraction \cite[Theorem~2.4]{Sokal-Zeng_22}.

Notice that we can use \cite[Theorem~2.2]{Sokal-Zeng_22}
to restate Theorem~\ref{thm.Deb}.

\addtocounter{theorem}{-3}
\begin{theorem}\hspace{-2mm}${{}^{\bf \prime}}$
The ordinary generating function of the polynomials\\ 
$\widetilde{Q}_n(x_1,x_2, \ytilde_1, y_2,u_1, u_2, \vtilde_1, v_2, \mathbf{w})$
defined in
\eqref{eq.def.Qntilde}
along with the substitution $\ytilde_1 = y_1, \vtilde_1=v_1$
has the same Jacobi-type continued fraction
as the polynomial
$Q_n^{(1)}(x_1,x_2,y_1,y_2,u_1,u_2,v_1,v_2,\mathbf{w})$
given in \cite[Theorem~2.2]{Sokal-Zeng_22}.
Therefore,
\be
\widetilde{Q}_n(x_1,x_2, y_1, y_2,u_1, u_2, v_1, v_2, \mathbf{w})
\;=\;
Q_n^{(1)}(x_1,x_2,y_1,y_2,u_1,u_2,v_1,v_2,\mathbf{w})\;.
\ee
\label{thm.Deb.prime}
\end{theorem}
\addtocounter{theorem}{2}

%
%
%

\subsection{Instances of the cyclic sieving phenomenon}
\label{subsec.csp.permutations}

We are now ready to state instances of the cyclic sieving phenomena on permutations 
with respect to the Corteel involution
or as a matter of fact, 
for any involution having $1$ fixed point when $n=0$ 
and $2^{n-1}$ fixed points when $n\geq 1$.

We first state a very general result in Proposition~\ref{prop.perm.csp}.

\begin{proposition}
Let $\Stat{} : \mathfrak{S}_n \to \mathbb{N}$ be a statistic defined on permutations.
In each of the following cases
the statistic $\Stat{}$ exhibits the cyclic sieving phenomenon with respect to
involutions that have $1$ fixed point when $n=0$ and $2^{n-1}$ fixed points when $n\geq 1$.

\begin{itemize}
	\item[(a)] $\Stat{}$ is obtained as a specialisation of the polynomial\\
	$Q^{(1)}_n(x_1,x_2, y_1, y_2,u_1, u_2, v_1, v_2, \mathbf{w})$
	defined in~\eqref{def.Qn}, 
	i.e.,
	there exists polynomials 
	$x_1(q)$, $x_2(q)$, $y_1(q)$,$y_2(q)$,$v_1(q)$,$v_2(q) \in\mathbb{R}[q]$
	and $w_k(q)\in \mathbb{R}[q]$ for $k\geq 0$
	satisfying
	\be
	\sum_{\sigma\in \mathfrak{S}_n} q^{\Stat{}} 
		\;=\;
	Q^{(1)}_n\biggl(
	x_1(q),x_2(q), y_1(q), y_2(q),u_1(q), u_2(q), v_1(q), v_2(q), (w_k(q))_{k\geq 0}
	\biggr)
	\ee
	for all $n\geq 0$.
	Further assume that the following equations are satisfied
	\begin{subeqnarray}
		x_1(-1) \cdot 
		y_1(-1)
		&=&
		1
		\\
		w_0(-1)
		&=&
		1
		\\
		x_2(-1) \,+\, y_2(-1) \,+\, w_1(-1)
		&=&
		1
		\\		
		\biggl[x_1(-1) \,+\, u_1(-1) \biggr]
		\times
		\biggl[y_1(-1) \,+\, v_1(-1) \biggr]
		&=&
		0
	\label{eq.csp.cond.perm}
	\end{subeqnarray}

	\item[(b)]
	$\Stat{}$ is obtained as a specialisation of the polynomial\\ 
	$\widetilde{Q}_n(x_1,x_2, y_1, y_2,u_1, u_2, v_1, v_2, \mathbf{w})$
	defined in~\eqref{eq.def.Qntilde}.
	Also assume that equations~\eqref{eq.csp.cond.perm}
	with 
	$y_1,v_1$ replaced by $\ytilde_1,\vtilde_1$
	are satisfied.

	\item[(c)] $\Stat{}$ is obtained as a specialisation of the polynomial\\ 
	   $Q_n^{(2)}(x_1,x_2,y_1,y_2,u_1,u_2,v_1,v_2,\mathbf{w},
   p_{+1},p_{+2},p_{-1},p_{-2},q_{+1},q_{+2},q_{-1},q_{-2},s)$\\
	defined in~\eqref{def.Qn.pq}, i.e.,
	there exists polynomials 
	$x_1(q)$, $x_2(q)$, $y_1(q)$, $y_2(q)$, $v_1(q)$, $v_2(q),$
	$p_{+1}(q)$,
	$p_{+2}(q)$, $p_{-1}(q)$, $p_{-2}(q)$, $q_{+1}(q)$,
	$q_{+2}(q)$,
	$q_{-1}(q)$, $q_{-2}(q)$, $s(q)$
	$\in\mathbb{R}[q]$
	and $w_k(q)\in \mathbb{R}[q]$ for $k\geq 0$
	such that 
	\begin{eqnarray}
	\sum_{\sigma\in \mathfrak{S}_n} q^{\Stat{}} &=&
		\nonumber\\
	&&
	\hspace*{-20mm}
	\Scale[0.6]{
Q^{(2)}_n\biggl(
	x_1(q),x_2(q), y_1(q), y_2(q),u_1(q), u_2(q), v_1(q), v_2(q), (w_k(q))_{k\geq 0},
	p_{+1}(q),
	p_{+2}(q), p_{-1}(q), p_{-2}(q), q_{+1}(q),
	q_{+2}(q),
	q_{-1}(q), q_{-2}(q), s(q)	
	\biggr)}
	\nonumber\\
	\end{eqnarray}
	for all $n\geq 0$.
	Further assume that the equations~(\ref{eq.csp.cond.perm}a,b)
	along with the following equations
	are satisfied
	\begin{subeqnarray}
		x_2(-1) \,+\, y_2(-1) \,+\, s(-1)\cdot w_1(-1)
		&=&
		1
		\\				
		\biggl[p_{-1}(-1)\cdot x_1(-1) \,+\, q_{-1}(-1)\cdot u_1(-1) \biggr]
		\times
		&&
				\nonumber\\
		\biggl[p_{+1}(-1)\cdot y_1(-1) \,+\,q_{+1}(-1)\cdot v_1(-1) \biggr]
		&=&
		0
	\label{eq.csp.cond.perm.pqgen}
	\end{subeqnarray}

\end{itemize}
\label{prop.perm.csp}
\end{proposition}

\begin{proof}
(a) We replace the variables in \cite[Theorem~2.2]{Sokal-Zeng_22} with  polynomials
$x_1(q)$, $x_2(q)$, $y_1(q)$, $y_2(q)$, $v_1(q)$, $v_2(q)$, $(w_k(q))_{k\geq 0}$ 
and then set $q=-1$.
Using the conditions in~\eqref{eq.csp.cond.perm}, 
it follows that the continued fraction~\cite[eq.~(2.24)\//(2.25)]{Sokal-Zeng_22}
under this substitution becomes~\eqref{eq.corteel.fp.contfrac}.


The proof of (b) and (c) is similar to that of (a)
but instead we use Theorem~\ref{thm.Deb} and
\cite[Theorem~2.7]{Sokal-Zeng_22}, respectively.
\end{proof}

\medskip

As a consequence of Proposition~\ref{prop.perm.csp},
we can state our main theorem for permutations as an easy corollary:

\begin{theorem}[CSP for permutations]
Let $\Stat, \Stat_\gamma,\Stat_\beta$ be three statistics on permutations satisfying $\Stat = \Stat_\beta+\Stat_\gamma$.
Assume that $\Stat_\gamma$ is one of the following statistics:
\begin{quote}
exclusive antirecord cycle double fall, exclusive record cycle double rise, 
cycle double rise, 
cycle double fall,\\
neither-record-antirecord fixed point,  pseudo-nesting.
\end{quote}
Also, assume that $\Stat_\beta$ is one of the following statistics:
\begin{quote}
neither-record antirecord cycle peaks, neither-record antirecord cycle valleys, cycle valley non-minima,\\
$\eareccpeak+\ereccval$, $\eareccpeak+\minval$\\
$\nrcpeak +\nrcval+\nrcdrise+\nrcdfall$,\\
lower crossings of type cpeak, lower crossings\\
lower nestings of type cpeak, lower nestings,\\
upper crossings of type cval, upper crossings, \\
upper nestings of type cval, upper nestings,\\
$\ucross + \lcross$, $\lnest + \unest$
\end{quote}

Then $\Stat$ exhibits the cyclic sieving phenomenon
with respect to
involutions that have~$1$ fixed point when $n=0$ and $2^{n-1}$ fixed points when $n\geq 1$
(in particular, with respect to the Corteel involution).
\label{thm.perm.csp}
\end{theorem}

We should mention that there are various other possibilites to exploit
Proposition~\ref{prop.perm.csp}.
We have only mentioned a few of these in Theorem~\ref{thm.perm.csp}.

As immediate consequences of Proposition~\ref{prop.perm.csp}
and Theorem~\ref{thm.perm.csp},
we will now re-establish several of the results in \cite[Section~4]{Adams_24}.

\begin{corollary}
The following statistics exhibit the cyclic sieving phenomenon with respect to 
involutions that have~$1$ fixed point when $n=0$ and $2^{n-1}$ fixed points when $n\geq 1$
(in particular, with respect to the Corteel involution).
\begin{itemize}
\item[(a)] \cite[Theorem~4.7]{Adams_24} The number of crossings (in the sense of Corteel)
\cite[Statistic~39]{findstat}; see (\ref{eq.cross.nest.Corteel}a).

\item[(b)] \cite[Corollary~4.8]{Adams_24} The number of nestings (in the sense of Corteel)
\cite[Statistic~223]{findstat}; see (\ref{eq.cross.nest.Corteel}b).

\item[(c)] \cite[Theorem~4.15]{Adams_24} The number of midpoints of decreasing subsequences of length 3 \cite[Statistic~371]{findstat}; 
in our terminology, these are neither-record-antirecords ($\nrar$).



\end{itemize}
\end{corollary}


\subsection{Some more instances of the cyclic sieving phenomenon}
\label{subsec.vincular.permutations}

We will now exhibit a few more instances of the cyclic sieving phenomenon 
for some statistics
which are not directly obtained from the general recipe in Proposition~\ref{prop.perm.csp}.
These will be the count of certain vincular patterns.
In particular, we will reprove \cite[Corollary~4.8]{Adams_24}.

\begin{theorem}
The generating function of the number of occurences of patterns 
$2-ab$ and $ab-2$ (these are all equidistributed)
have the following Stieltjes-type continued fraction:
\begin{equation}
\sum_{n=0}^\infty \sum_{\sigma \in\mathfrak{S}_n} q^{ab-2(\sigma)}t^n
\;=\;
\sum_{n=0}^\infty \sum_{\sigma \in\mathfrak{S}_n} q^{2-ab(\sigma)}t^n
\;=\;
			\cfrac{1}{1 - \cfrac{ [1]_q \,t}{1 -   \cfrac{[1]_q t}{1 -  \cfrac{[2]_q t}{1 - \cfrac{[2]_q}{1-\cdots}}}}}
                \;.
		\label{eq.perm.vincular.cfrac}
\end{equation}
Hence, they exhibit the cyclic sieving phenomenon with respect to involutions that have $1$ fixed point when $n=0$ and $2^{n-1}$ fixed points when $n\geq 1$.
\label{thm.perm.vincular}
\end{theorem}

\begin{proof}
It is known that the generating function of the pattern $2-13$ has this Stieltjes-type continued fraction
\cite[Corollary~23]{Claesson_02}.
Also, it is known that all patterns of the form $2-ab$ and $ab-2$ are equidistributed
\cite[Proposition~1]{Claesson_01}, see also \cite[Proposition~2]{Claesson_02}.
This proves the identity~\reff{eq.perm.vincular.cfrac}.
Plugging in $q=-1$ in~\reff{eq.perm.vincular.cfrac} gives us
\be
\cfrac{1}{1-\cfrac{t}{1-t}}
\;=\;
1+\sum_{n=1}^\infty 2^{n-1}
\ee 
which is what we want.
\end{proof}

\begin{remark}
1. 
Theorem~\ref{thm.perm.vincular} was already proved for patterns $13-2$ and $31-2$ 
in~\cite[Corollary~4.8]{Adams_24}; 
we have now provided a new proof of this result.

2. 
We remark that in a recent work, Bitonti et al \cite{Bitonti-Deb-Sokal_tfrac} provide
a vast multivariate generalisation of the continued fraction~\reff{eq.perm.vincular.cfrac},
see~\cite[Theorems~7.2/3.5]{Bitonti-Deb-Sokal_tfrac}.

3.
The vincular patterns $32-1$ or $12-3$,
which have been studied in \cite[Theorem~4.20, Corollary~4.22]{Adams_24}, 
do not have any known nice classical continued fractions.
Therefore, we do not know (at present) how to use the techniques in the current paper to 
prove these results.

\end{remark}

\subsection{Equidistribution of wexx and cdes: Proof of 
\cite[Conjecture~4.24]{Adams_24}}
\label{sec.Adams.conj}

In this section, we will prove 
\cite[Conjecture~4.24]{Adams_24}
which has been stated in
Conjecture~\ref{conj.equidist};
we will show that the statistic $\wexx(\sigma)$, defined in~\eqref{eq.def.wexx},
is equidistributed with $\cdes(\sigma)$, 
the cycle descent number defined in~\eqref{eq.def.cdes}.
We do this by showing that the ordinary generating functions for permutations counted with respect to either statistic
give the same continued fraction.
We obtain the former statistic by using a specialization of the {\em first J-fraction for permutations} of
Sokal and Zeng \cite[Theorem~2.2]{Sokal-Zeng_22},
and we obtain the latter statistic from our Theorem~\ref{thm.Deb},
which is a variant of Deb's continued fraction \cite[Theorem~3.1]{Deb_23}.

Our first ingredient is the following lemma which expresses these two statistics 
as combinations of some statistics in \cite[Section~2.5]{Deb_thesis}.

\begin{lemma}

\begin{itemize}
\item[(a)] 
The statistic $\wexx(\sigma)$, 
which counts 
the number of weak excedances of $\sigma$ that are also mid-points of 
a decreasing subsequence of length~3, 
can be expressed as
\be
	\wexx(\sigma) \;=\; \nrfix(\sigma) + \nrcval(\sigma)+ \nrcdrise(\sigma)\;.
\label{eq.wexx.relation}
\ee

\item[(b)] 
The statistic $\cdes(\sigma)$, which is  the cycle descent number of $\sigma$, 
can be expressed as
\be
 \cdes(\sigma) \;=\; \nminval(\sigma) + \cdfall(\sigma)\;.
\label{eq.cdes.relation}
\ee
\end{itemize}
\label{lem.wexx.cdes.relations}
\end{lemma}

\begin{proof} 
(a) 
From the cycle classification for permutations, we know that a weak-excedance 
can either be a fixed point, a cycle valley or a cycle double rise.
Next, we know that an index $i\in [n]$ 
is the midpoint of a decreasing subsequence of length 3
if and only if it is a neither-record-antirecord ($\nrar$).
Using these two pieces of information simultaneously 
in the record-and-cycle classification
gives us~\eqref{eq.wexx.relation}.

\medskip

(b) First replace $i$ with $\sigma^{-1}(i)$ in~\eqref{eq.def.cdes}
to notice that $\cdes(\sigma)$ is equal to 
the number of indices $i\in[n]$ such that $\sigma^{-1}(i) > i$ 
where $i$ is not the smallest element of a cycle.
Such an index $i$ could either be a cycle double fall or a
cycle valley that is not the smallest in its cycle 
(a cycle double fall can never be the smallest element of a cycle),
or in other words
\begin{equation}
 \cdes(\sigma) \;=\; \nminval(\sigma) + \cdfall(\sigma)
\end{equation}
which is what we want.
\end{proof}

Let $P_n(x)$ and $Q_n(x)$ be the ordinary generating polynomials of the statistics
$\wexx$ and $\cdes$, respectively, i.e.,
\begin{eqnarray}
	P_n(x) \; \eqdef \; \sum_{\sigma\in \mathfrak{S}_n} x^{\wexx(\sigma)}\;, \qquad\qquad 
	Q_n(x) \; \eqdef \; \sum_{\sigma\in \mathfrak{S}_n} x^{\cdes(\sigma)}\;.
\end{eqnarray}
The following theorem is now a straightforward consequence of Lemma~\ref{lem.wexx.cdes.relations} 
applied to  \cite[Theorem~2.2]{Sokal-Zeng_22} and Theorem~\ref{thm.Deb}.

\begin{theorem}
	The ordinary generating functions $\sum_{n=0}^{\infty} P_n(x) t^n$, 
	and $\sum_{n=0}^{\infty} Q_n(x) t^n$, 
	both have the following Jacobi-type continued fraction expansion
	\begin{equation}
	\sum_{n=0}^{\infty} P_n(x) t^n \;=\; \sum_{n=0}^{\infty} Q_n(x) t^n
		\;=\;
	\cfrac{1}{1 - 1\, t - \cfrac{ 1\cdot 1 \,t^2}{1 -  (2\!+\!x) t - \cfrac{2(1\!+\!x) t^2}{1 - (3\!+\!2x)t - \cfrac{3(1\!+\!2x) t^2}{1 - \cdots}}}}
		\label{eq.equidist}
	\end{equation}
	with coefficients
	\be
	\beta_{n} \;=\; n(n-1+x)\;,
	\qquad 
	\gamma_{n} \;=\; n + (n-1)x\;.
	\label{eq.equidist.coeff}
	\ee
	Thus, $P_n(x) = Q_n(x)$ for all $n\geq 0$.
	Hence, the statistics $\wexx$ and $\cdes$ are equidistributed \cite[Conjecture~4.24]{Adams_24}.
	\label{thm.equidist}
\end{theorem}

\begin{proof}
From Lemma~\ref{lem.wexx.cdes.relations}(a),
we get that applying the substitions
$v_1 = v_2 = x$, $w_n=x$ for $n>0$
and setting all other parameters to $1$
to the polynomials $Q_n^{(1)}$, recalled in~\eqref{def.Qn},
gives us $P_n(x)$.
The same substitutions applied to the continued fraction 
in the first J-fraction for permutations of Sokal and Zeng \cite[Theorem~2.2]{Sokal-Zeng_22}
gives us our continued fraction in~\eqref{eq.equidist}/\eqref{eq.equidist.coeff}.

For the polynomials $Q_n(x)$ enumerating $\cdes$,
we instead 
infer from Lemma~\ref{lem.wexx.cdes.relations}(b)
that applying the following substitutions to Theorem~\ref{thm.Deb} will give us the desired result:
substitute $v_1 = x_2 = u_2 = x$
and set all other variables to $1$.
\end{proof}

Finally, notice that setting $x=-1$ to the continued fraction~\eqref{eq.equidist}/\eqref{eq.equidist.coeff}
gives us~\eqref{eq.corteel.fp.contfrac}. 
As a consequence, we have the following corollary:

\begin{corollary}
The statistics $\wexx$ \cite[Statistic~373]{findstat} and 
$\cdes$ \cite[Statistic~317]{findstat}
exhibits the cyclic sieving phenomenon with respect to
involutions that have~$1$ fixed point when $n=0$ and $2^{n-1}$ fixed points when $n\geq 1$
(in particular, with respect to the Corteel involution).
\label{cor.equidist.csp}
\end{corollary}

We remark that Corollary~\ref{cor.equidist.csp} for the statistic $\cdes$ was established in
\cite[Theorem~4.10]{Adams_24}.

\subsection{Cyclic sieving phenomenon for Statistic~123: 
Proof of \cite[Conjecture~4.23]{Adams_24}}
\label{sec.stat123}

In this section, we will prove \cite[Conjecture~4.23]{Adams_24}
which asserts that \cite[Statistic~123]{findstat} exhibits the cyclic sieving phenomenon 
with respect to the Corteel involution.
The original description of this statistic involved the Simion--Schmidt map \cite{Simion_85};
we will begin this section by first recalling this map.
After this, we provide an alternate and better description of this statistic~123.


Let $\pi\in \mathfrak{S}_3$, be a permutation on 3 letters.
The Simion--Schmidt map is a map $\Phi_\pi:\mathfrak{S}_n \to \mathfrak{S}_n$
with respect to the permutation pattern $\pi$;
the map $\Phi_\pi$ satisfies the property that $\Phi_\pi(\sigma)$ has no occurrence of the pattern $\pi$.
In this paper, we will only work with $\Phi_{321}$, and hence, we now provide its explicit description here:\\
Keep the records (left-to-right maxima) of $\sigma$ fixed, 
and then sort all the other entries of $\sigma$ in increasing order and put them in the remaining positions.
The obtained permutation is $\Phi_{\pi}(\sigma)$.
For example, 
consider 
\be
\sigma  = \underline{9}\,3 \, 7 \,4 \,6 \, \underline{11} \, 5 \, 8 \,  10 \, 1 \, 2\;;
\ee
here the records are underlined.
Its image under the Simion--Schmidt map  $\Phi_{321}$ is
\be
\Phi_{321}(\sigma)
\;=\;
 \underline{9}\,1 \, 2 \,3 \,4 \, \underline{11} \, 5 \, 6 \,  7 \, 8 \, 10\;.
\ee
Here, the map $\Phi_{321}$ turns each occurrence of the pattern $321$ in $\sigma$
into an occurrence of $312$ in $\Phi_{321}(\sigma)$.

The original definition of \cite[{Statistic~123}]{findstat} is as follows:
\be
\Stat{123}(\sigma)
\;\eqdef\;
\inv(\sigma) \,-\, \inv(\Phi_{321}(\sigma))
\ee
where recall that $\inv$ counts the number of inversions.

To prove \cite[Conjecture~4.23]{Adams_24}, we will provide an alternate description of~$\Stat{123}$
in terms of {\em positional marked patterns} \cite{Thamrongpairoj_22}.
Consider a tuple $(\pi, I)$ where $\pi\in \mathfrak{S}_m$ and 
$I = \{i_1, i_2, \ldots, i_k \} \subseteq [m]$ is a list of indices of $\pi$ which are to be circled.
Using the tuple $(\pi, I)$, we define the following statistic on permutations:
\begin{eqnarray}
\pmp_{(\pi, I)}(\sigma)
	&\eqdef &
\#\{1\leq j_1<j_2<\ldots< j_k\leq n \: | \:\exists 1\leq \bar{j_1}<\bar{j_2}<\ldots< \bar{j_m}\leq n 
\nonumber\\
	&&\quad\text{ such that 
the subword } \sigma_{\bar{j_1}}\sigma_{\bar{j_2}}\ldots\sigma_{\bar{j_k}} 
\text{is an occurrence}
\nonumber\\
	&&\quad	\text{of the pattern $\pi$ that satisfies }
\bar{j_{i_1}} = j_1, \ldots, \bar{j_{i_m}} = j_m
\}\;.
\label{eq.def.pmp}
\end{eqnarray}
In their paper, Remmel and Thamrongpairoj \cite{Thamrongpairoj_22} only considered the situation where $|I| = 1$;
they only studied positional marked patterns with respect to one position.
However, in our paper, we will study the tuple $(321,\{2,3\})$.
Following the notation in \cite{Adams_24}, 
we will denote this as $3\,\circlesign{2}\, \circlesign{1}$.
For the convenience of the reader, we restate~\eqref{eq.def.pmp}
when $(\pi, I) = 3\,\circlesign{2}\, \circlesign{1}$
\begin{eqnarray}
\pmp_{{\tiny 3\,\circlesignsubscript{2}\, \circlesignsubscript{1}}}(\sigma)
&=&
\#\{(j,k) \: |\: 
\text{ there exists }
i \text{ satisfying } 1\leq i<j<k\leq n 
\nonumber\\
	&&\quad\text{ such that } \sigma_i > \sigma_j > \sigma_k \}\;.
\end{eqnarray}

We are now ready to state an alternate description of $\Stat{123}$:

\begin{proposition}
The following equality holds:
\be
	\Stat{123}(\sigma) 
	\;=\;
	\pmp_{{\tiny 3\,\circlesignsubscript{2}\, \circlesignsubscript{1}}}(\sigma)
	\;.
\ee
\label{prop.alt.stat123}
\end{proposition}

In order to prove this Proposition, we require the following lemma:

\begin{lemma}
Let $\sigma\in \mathfrak{S}_n$ be a permutation that has its left-to-right maxima 
in positions $I = \{i_1,i_2,\ldots, i_k\}$.
Then the inversions contributed by the left-to-right maxima as the larger value is given by 
\be
\#\{(i,j) \: | \: \sigma_i > \sigma_j \text{ and } i\in I\}
\;=\;
	(\sigma_{i_1} -1)
	\,+\,
	(\sigma_{i_2} -2)
	\,+\,
	\cdots
	\,+\,
	(\sigma_{i_k} - k)
	\;.
\label{eq.inv.lrmax}
\ee
\label{lem.inv.lrmax}
\end{lemma}

We leave the proof of Lemma~\ref{lem.inv.lrmax} to the reader.
The important thing to notice is that~\eqref{eq.inv.lrmax} only depends on the positions and values of the left-to-right maxima.
We now proceed to prove Proposition~\ref{prop.alt.stat123}.

\begin{proof}[Proof of Proposition~\ref{prop.alt.stat123}]
Consider a pair of indices $(i,j)$ that forms an inversion in $\Phi_{321}(\sigma)$.
From the description of $\Phi_{321}$, we see that for this to be true,
$i$ must be the position of a left-to-right maxima in $\sigma$.
Thus, from Lemma~\ref{lem.inv.lrmax} it follows that $\inv(\Phi_{321}(\sigma))$
is counted by~\eqref{eq.inv.lrmax},
which is also the number of inversions of $\sigma$ contributed by its left-to-right maxima as the larger value.
From this it follows that 
\begin{eqnarray}
	\Stat{123}(\sigma)
	&=&
	\inv(\sigma) \,-\, \inv(\Phi_{321}(\sigma)) 
	\nonumber\\
	&=&
	\#\{(i,j) \: | \: \sigma_i > \sigma_j \text{ and } i \text{ is not the position of a left-to-right maximum}\}\;.
	\nonumber\\
\end{eqnarray}
However, this means that
for every such position $i$ of the permutation $\sigma$, there exists a position $i_0<i$ 
such that $i_0$ is the position of a left-to-right maximum.
The triple $i_0<i<j$ is an occurrence of the pattern $321$,
and thus, $\Stat{123}(\sigma)$ counts the occurrences of pairs which correspond to
$2$ and $1$ positions in all such occurrences.
\end{proof}

Having established Proposition~\ref{prop.alt.stat123},
we now prove \cite[{Conjecture~4.23}]{Adams_24}.

\begin{theorem}[{\cite[{Conjecture~4.23}]{Adams_24}}]
The difference in the Coxeter length of a permutation and its image under the Simion--Schmidt map $\Phi_{321}$
\cite[{Statistic~123}]{findstat}
exhibits the cyclic sieving phenomenon under involutions with 
$1$ fixed point when $n=0$ and $2^{n-1}$ fixed points for all $n\geq 1$.
\end{theorem}

\begin{proof}
From Proposition~\ref{prop.alt.stat123}, we know that $\Stat{123}(\sigma)$ 
is the number of positional marked patterns $3\,\circlesign{2}\, \circlesign{1}$ in the permutation $\sigma$.
With this translation, 
we see that the involution $\Psi$ defined in \cite[{Proof of~Theorem~4.15}]{Adams_24},
that has $1$ fixed point when $n=0$ and $2^{n-1}$ fixed points for all $n\geq 1$, 
satisfies the following property:
\be
\left|\pmp_{{\tiny 3\,\circlesignsubscript{2}\, \circlesignsubscript{1}}}(\sigma)
\,-\,
\pmp_{{\tiny 3\,\circlesignsubscript{2}\, \circlesignsubscript{1}}}(\Psi(\sigma))\right|
 \;=\; \pm 1
\ee
This finishes the proof.
\end{proof}

\begin{remark}
After discovering the equivalence in Proposition~\ref{prop.alt.stat123},
we updated the webpage of the statistic 123 on the FindStat website.
It now displays the interpretation in terms of positional marked patterns $3\,\circlesign{2}\, \circlesign{1}$.

We also remark that there is no nice Jacobi-type continued fraction for this statistic. 
Hence, the techniques used in the rest of the paper do not work for this statistic.
\end{remark}

\section{Set partitions}
\label{sec.setpart}

{\em A set partition of $[n]$} is a partition of the set $[n]$ into nonempty blocks;
it is counted by the Bell number $B_n$.
We use $\Pi_n$ to denote the set of all set-partitions of $[n]$.
For a partition $\pi$, $|\pi|$ will denote its number of blocks.
An important property of the Bell numbers $B_n$ is that they have the following J-fraction \cite{Touchard_56, Flajolet_80}:

\be
\sum_{n=0}^\infty B_n  t^n
\;=\;
\cfrac{1}{1 - t - \cfrac{t^2}{1- 2t - \cfrac{2 t^2}{1- 3t - \cfrac{3 t^2}{1- \ldots}}}	}\;.
\ee
We will use various generalisations of this identity in this section.

In Section~\ref{sec.setpart.stats} we define some statistics on set partitions.
Then in Section~\ref{sec.KZ.setpart} we recall the Kasraoui--Zeng
involution \cite{Kasraoui_06}
in terms of labelled Motzkin paths and count its number of fixed points.
After this, in Section~\ref{sec.cddsy.setpart}, 
we look at Chen et al's involution \cite{Chen_07}
and count its number of fixed points.
Then in Section~\ref{sec.SZpoly.setpart},
we will use the statistics defined in Section~\ref{sec.setpart.stats}
to recall multivariate generating polynomials
introduced by Sokal and Zeng \cite{Sokal-Zeng_22}.
Finally, we state our main results 
in Section~\ref{sec.csp.setpart}.
Finally in Section~\ref{sec.mahonian},
we conclude with a discussion on Mahonian statistics on set partitions.

\subsection{Statistics on set partitions}
\label{sec.setpart.stats}

Let $\pi\in \Pi_n$ be a set partition of $[n]$.
Following Sokal and Zeng \cite{Sokal-Zeng_22},
we first classify each element $i\in [n]$ as follows:
\begin{itemize}
	\item $i$ is an {\em opener} ($\op$) if it is the smallest element of a block of size $\geq 2$;

	\item $i$ is a {\em closer} ($m_{\ge 2}$) if it is the largest element of a block of size $\geq 2$;

	\item $i$ is an {\em insider} ($\inside$) if it is a non-opener non-closer  element of a block of size $\geq 3$;

	\item $i$ is a {\em singleton} ($m_1$) if it is the sole element of a  block of size $1$.
\end{itemize}
Clearly every element $i$ belongs to precisely one of these four classes;
we will call this the {\bf{\em type classification for set partitions}}\footnote{We borrow 
the term {\em type} from Kasraoui and Zeng \cite{Kasraoui_06};
their usage of the word type is different but in a similar context.}.


We say that an index $i$ is an {\em exclusive record} of $\pi$
if it is not the largest element in its block,
(that is, it is either an opener or an insider)
and its right neighbor (within its block) sticks out farther
to the right than any right neighbor (within its block) of a vertex $<i$.
With this notion of an exclusive record, we refine openers and insiders as follows:
\begin{itemize}
	\item $i$ is an {\em exclusive record opener} ($\erecop$) if it is an opener 
		and also an exclusive record;
	\item $i$ is an {\em non-exclusive record opener} ($\nerecop$) if it is an opener 
		but not an exclusive record;

	\item $i$ is a {\em exclusive record closer} ($\erecin$) if it is a closer and also an exclusive record;
        
	\item $i$ is a {\em non-exclusive record closer} ($\nerecin$) if it is a closer but not an exclusive record.

\end{itemize}
These four statistics along with insiders and singletons
form 6 disjoint categories and we will call them the {\bf{\em record-and-type classification}}.

\bigskip

After this, we introduce crossings, nestings, overlaps and coverings on set partitions.
But first, we recall the {\em partition graph} of $\pi$;
we follow the convention of Sokal and Zeng.
We associate to the partition $\pi$ a graph $\mathcal{G}_\pi$ with vertex set $[n]$
such that $i,j$ are joined by an edge if and only if they are consecutive elements within the same block.
(Thus, the graph $\mathcal{G}_\pi$ has $|\pi|$ connected components and $n-|\pi|$ edges.)
With this in mind, we say that a quadruplet $i<j<k<l$ forms a 
\begin{itemize}
	\item {\em crossing} ($\cros$) if $(i,k)$ and $(j,l)$ are edges in $\mathcal{G}_\pi$;

	\item {\em nesting} ($\nes$) if $(i,l)$ and $(j,k)$ are edges in $\mathcal{G}_\pi$.
\end{itemize}
Also, we say that a triple $i<j<l$ forms a 
\begin{itemize}
	\item {\em pseudo-nesting} ($\psne$) if $j$ is a singleton and $(i,l)$ is an edge in $\mathcal{G}_\pi$.
\end{itemize}

Sokal and Zeng further refine the categories of crossing and nesting depending on the type of the second index $j$
and introduce the following four classes:
{\em crossing of opener type} ($\crop$),
{\em crossing of insider type} ($\crin$),
{\em nesting of opener type} ($\neop$),
{\em nesting of insider type} ($\nein$).

Let $B_1,B_2$ be blocks of $\pi$. We say that the pair $(B_1,B_2)$ forms
\begin{itemize}
	\item an {\em overlap} ($\ov$) if $\min B_1 < \min B_2 < \max B_1 < \max B_2$;

	\item a {\em covering} ($\cov$) if $\min B_1 < \min B_2 < \max B_2 < \max B_1$;

	\item a {\em pseudo-covering} ($\pscov$) if $\min B_1 < \min B_2 = \max B_2 < \max B_1$
		     (so that here $B_2$ is a singleton).
\end{itemize}
They also define two related quantities $\ovin$ and $\covin$ as follows:
\begin{subeqnarray}
   \ovin(\pi)
   & = &
   \!\!
   \sum\limits_{j \in {\rm insiders}\cap B_2}  \!\!
   \#\{ (B_1,B_2) \colon\:  j \in B_2 \,\hbox{ and }\,
                          \min B_1 < j < \max B_1 < \max B_2  \}
         \nonumber \\[-4mm] \\[1mm]
   \covin(\pi)
   & = &
   \!\!
   \sum\limits_{j \in {\rm insiders}\cap B_2}  \!\!
   \#\{ (B_1,B_2) \colon\:  j \in B_2 \,\hbox{ and }\,
                          \min B_1 < j < \max B_2 < \max B_1  \}
         \nonumber \\[-4mm]
 \label{def.ovin.covin}
\end{subeqnarray}

We also recall the notion of a {\em block record}.
We say that $j$ is a {\em block record} if it is either an opener or an
insider and its block sticks out farther to the right than any block containing a vertex $< j$.
We write 
\begin{itemize}
	\item $\brecin(\pi)$ for the number of insiders that are block-records,

	\item $\nbrecin(\pi)$ for the number of insiders that are not block-records,

	\item $\brecop(\pi)$ for the number of openers that are block-records,

	\item $\nbrecop(\pi)$ for the number of openers that are not block-records.
\end{itemize}

\subsection{The Kasraoui--Zeng involution on set partitions}
\label{sec.KZ.setpart}

In \cite{Kasraoui_06}, Kasraoui and Zeng constructed an involution on the set $\Pi_n$
that exchanges the number of crossings and nestings while keeping various other statistics fixed.
We recall this involution and then count its number of fixed points.
However, we will state this involution via labelled Motzkin paths ({\em Charlier diagrams of Viennot})
to be consistent with our formalism in the previous sections.

The Kasraoui--Zeng bijection (as stated by Sokal and Zeng \cite[sec.~7.2]{Sokal-Zeng_22})
is a correspondence between $\Pi_{n}$ and the set of
$(\bfscra, \bfscrb,\bfscrc)$-labelled Motzkin paths of length $n$, 
where
the labels $\xi_i$ lie in the sets
\begin{subeqnarray}
   \scra_h        & = &  \{0 \}  \qquad\quad\;\;\hbox{for $h \ge 0$}  \\
   \scrb_h        & = &  \{0,\ldots,  h-1 \}    \quad\hbox{for $h \ge 1$}  \\
   \scrc_0    &=& \{2\}\times \{0\}\\
   \scrc_h        & = &  \{1\}\times \{0,1,\ldots, h-1\}  \cup \{2\}\times \{0\} \quad\hbox{for $h \ge 1$}
 \label{def.abc.setpart}
\end{subeqnarray}
Let $\motzkinset_{n}$ denote this set of $(\bfscra, \bfscrb,\bfscrc)$-labelled
Motzkin paths of length $n$ with label sets~(\ref{def.abc.setpart}a,b,c).
These are the {\em Charlier diagrams} of Viennot \cite{Viennot_83}.
The Kasraoui--Zeng bijection as per \cite[Section~7.2]{Sokal-Zeng_22}
is a bijective map $\SZbij_{n} : \Pi_{n} \to \motzkinset_{n}$.

We say that a level step at height $h$ is of the {\em first kind} ({\em second kind} resp.)
if the first index of $\xi_i\in \scrc_h$ is $1$ ($2$).


We will now describe an involution $\phi_{n}: \motzkinset_{n} \to \motzkinset_{n}$;
the Kasraoui--Zeng involution will then be obtained by the composition of bijections 
$\bigl(\SZbij_{n}\bigr)^{-1} \circ \phi_{n} \circ \SZbij_{n}$. 
Let $(\omega,\xi)\in \motzkinset_{n}$
where $\xi = (\xi_1,\ldots,\xi_{n})$.
Let $\xicomplement = (\xicomplement_1,\ldots, \xicomplement_{n})$ be the sequence where
$\xicomplement_i$
is given by
\be
\xicomplement_i
\;=\;
\begin{cases}
        \xi_i = 0 &\qquad \text{if step $i$ is a rise}\\
         h-1 - \xi_i & \qquad \text{if step $i$ is a fall}\\
	(1,h-1-\xi_i) &\qquad \text{if step $i$ is a level step of the first kind}\\
	(2,\xi_i) = (2,0) &\qquad \text{if step $i$ is a level step of the second kind}
\end{cases}
\ee

It clearly follows that $(\omega,\xicomplement)\in \motzkinset_{n}$
and that $(\xicomplement)^{{\rm c}} = \xi$.
Thus, the map $(\omega, \xi)\mapsto (\omega,\xicomplement)$
is an involution on the set $\motzkinset_{n}$;
we use $\phi_{n}$ to denote this involution.
The Kasraoui--Zeng involution on set partitions
is the map $\bigl(\SZbij_{n}\bigr)^{-1} \circ \phi_{n} \circ \SZbij_{n}$.

An important property of this involution is that crossings and nestings are exchanged;
i.e., for a partition $\pi$ with $\widehat{\pi} \coloneqq \biggl(\bigl(\SZbij_{n}\bigr)^{-1} \circ \phi_{n} \circ \SZbij_{n}\biggr)(\pi)$, 
we have
\be
(\cros(\pi), \nes(\pi))
\;=\;
(\nes(\widehat{\pi}), \cros(\widehat{\pi})) \;.
\ee

\bigskip

We now establish the number of fixed points of the Kasraoui--Zeng involution.
First, let $(a_n)_{n\geq 0}$ be the sequence of integers given by the linear recurrence
\be
a_{n} \;=\; 3 a_{n-1} \,-\, a_{n-2}
\ee
with initial conditions $a_{0} = a_1 = 1$, see \cite[A001519]{OEIS}.
In fact, these numbers are the odd subsequence of the Fibonacci numbers and hence we will call them the {\em odd Fibonacci numbers}.
The first few numbers of this sequence are
\be
1, 1, 2, 5, 13, 34, 89, 233, 610, 1597, 4181, 10946, 28657, \ldots
\ee
In the rest of this paper, we will use $F_{2n-1}\coloneqq a_n$ to denote these numbers.
(Here $F_{-1} = 1$.)

\begin{lemma}
The number of fixed points of $\Pi_n$ under the action of the Kasraoui--Zeng involution 
$\bigl(\SZbij_{n}\bigr)^{-1} \circ \phi_{n} \circ \SZbij_{n}$
is the odd Fibonacci number $F_{2n-1}$ \cite[A001519]{OEIS}.
\label{lem.csp.fixed.KZ}
\end{lemma}

\begin{proof}
It suffices to count the number of labels $\xi$ such that 
$\xi = \xicomplement$. 
Using general theory,
we obtain that the generating function of labelled Motzkin paths that are fixed points of the involution $\phi_n$ 
is given by the Jacobi-type continued fraction:
\be
\cfrac{1}{1 - t - \cfrac{t^2}{1- 2t }}\;.
\label{eq.odd.Fibonacci.contfrac}
\ee
This finishes the proof.
\end{proof}



%
%

\subsection{The Chen--Deng--Du--Stanley--Yan (CDDSY) involution on set partitions}
\label{sec.cddsy.setpart}

In \cite{Chen_07}, Chen et al gave a different involution on set partitions.
In their paper, the authors introduced the notion of
{\em vacillating tableaux}, which are certain walks on the Young's lattice,
and they constructed a bijection between set partitions and vacillating tableaux.
They then pullback a simple involution on vacillating tableaux via this bijection,
to describe their involution on set partitions.
We will refer to this involution as the {\em CDDSY involution}.
In this section, we will show that the number of fixed points of 
the CDDSY involution is also given by Lemma~\ref{lem.csp.fixed.KZ}.

We first recall some basic definitions.
Let $\mathbb{Y}$ be the {\em Young's lattice}, that is, the set of all integer partitions ordered component-wise.

\begin{definition}
A {\em vacillating tableaux} $V_{\lambda}^{2n}$ of shape $\lambda$ and length $2n$ is a sequence
$\emptyset=\lambda^0,\lambda^1,\lambda^2,\ldots, \lambda^{2n} = \lambda$ of integer partitions
with the property that $\lambda^{2i+1}$ is obtained by
either setting $\lambda^{2i+1} = \lambda^{2i}$ or by deleting a square from  $\lambda^{2i}$,
and $\lambda^{2i}$ is obtained by either setting
$\lambda^{2i} = \lambda^{2i-1}$ or by adding a square to $\lambda^{2i-1}$.
\end{definition}

In terms of the $U, D$ and $I$ operators on the Young's lattice $\mathbb{Y}$,
we can restate the conditions in this definition as follows:
given a $\lambda^{2i-2}$ in the sequence of a vacillating tableaux,
the integer partition $\lambda^{2i-1}$ is obtained by applying the operator 
$I + D$ to $\lambda^{2i-2}$ and then picking one of the summands,
and similarly, 
the integer partition $\lambda^{2i}$ is obtained by applying the operator 
$I + U$ to $\lambda^{2i-1}$  and then selecting one of the summands.

Let $\mathcal{VT}_{n}$ denote the set of all vacillating tableaux of shape $\emptyset$
and length $2n$. 
In \cite[Section~3]{Chen_07}, the authors constructed a bijection 
$\CDDSY_n : \Pi_n \to \VT_n$. 
%
%
%
%
%
%
Let $\tau_n : \VT_n \to \VT_n$ be the involution defined by taking the conjugate to each partition
$\lambda^i$.
The CDDSY involution on set partitions (defined in \cite[proof of Theorem~1]{Chen_07})
is the map 
$\bigl(\CDDSY_n\bigr)^{-1} \circ \tau_n \circ \CDDSY_n$.

We now establish the number of fixed points of the CDDSY involution.

\begin{lemma}
The number of fixed points of $\Pi_n$ under the action of the CDDSY involution 
$\bigl(\CDDSY_n\bigr)^{-1} \circ \tau_n \circ \CDDSY_n$
is the odd Fibonacci number $F_{2n-1}$ \cite[A001519]{OEIS}.
\label{lem.csp.fixed.CDDSY}
\end{lemma}


\begin{proof}
It suffices to count the number of fixed points of the involution $\tau_n$.
Let $\VT_n^{\flat}$ denote the set of all vacillating tableau $(\lambda^0, \ldots, \lambda^{2n})$
such that each integer partition $\lambda^{i}$ is either $\emptyset$ or $1$.
We will show that $\VT_n^{\flat}$ is the set of fixed points of $\tau_n$.
In fact, it is immediate that an element of $\VT_n^{\flat}$ is always a fixed point of $\tau_n$.
On the other hand, any vacillating tableaux not in the set $\VT_n^{\flat}$
must have index $i$ such that
$\lambda^i\vdash 2$.
For such an index $i$, the partition $\lambda^i$ and its conjugate must be different.
Hence, such a vacillating tableaux cannot be a fixed point of $\tau_n$.

Let $M_n^\flat$ denote the set of Motzkin paths of length $n$ which never cross the line $x=1$,
where the level steps at height $1$ gets a label $0$ or $1$.
From \cite{Flajolet_80}, we know that the ordinary generating function of the 
cardinalities of $M_n^\flat$ is given by the continued fraction
\be
\sum_{n=0}^\infty \bigl|M_n^\flat\bigr| t^n
\;=\;
\cfrac{1}{1 - t - \cfrac{t^2}{1- 2t }}\;.
\ee
Notice that the right-hand side is the same as in~\eqref{eq.odd.Fibonacci.contfrac}.
Thus, a bijection between the sets $\VT_n^{\flat}$ and $M_n^\flat$
will prove this lemma.
We construct such a bijection now.

\medskip


Let $(\lambda^0, \ldots, \lambda^{2n})\in \VT_n^\flat$ and
let $\omega_i\in \mathbb{N}\times \{0,1\}$ be the pair
\be
\omega_i 
\;=\;
\begin{cases}
(i,0) \qquad \text{if $\lambda^{2i} = \emptyset$}
\\
(i,1) \qquad \text{if $\lambda^{2i} = 1$}
\end{cases}
\;.
\ee
The sequence $\omega=(\omega_0,\omega_1,\ldots,\omega_n)$
is a path of length $n$ which never crosses the line $x=1$.
By construction, step $i$ is a level step at height $1$ if and only if 
$\lambda^{2i-2} = \lambda^{2i} =1$.
In this situation, the integer partition $\lambda^{2i-1}$ can be either $1$ or $\emptyset$,
and step $i$ gets label $0$ and $1$, respectively.
This gives us a Motzkin path $\omega$ where each level step at height $1$ 
gets a label $0$ or $1$.
This is our desired bijection.
\end{proof}
%
%

\subsection{Sokal and Zeng's multivariate polynomials enumerating statistics on set partitions}
\label{sec.SZpoly.setpart}

In this section, we recall three of Sokal--Zeng's multivariate polynomials enumerating statistics on set partitions;
first we shall state \cite[eq.~(3.6)]{Sokal-Zeng_22} which does not have any $p,q$-statistics,
then we shall state their first $p,q$-generalisation involving crossings and nestings 
\cite[eq.~(3.11)]{Sokal-Zeng_22},
and then finally we state their second $p,q$-generalisation involving overlaps and coverings.

The first polynomial is \cite[eq.~(3.6)]{Sokal-Zeng_22}
\be
   B_n^{(1)}(x_1,x_2,y_1,y_2,v_1,v_2)
   \;=\;
   \sum_{\pi \in \Pi_n}
      x_1^{m_1(\pi)} x_2^{m_{\ge 2}(\pi)}
      y_1^{\erecin(\pi)} y_2^{\erecop(\pi)}
      v_1^{\nerecin(\pi)} v_2^{\nerecop(\pi)}
   \;.
 \label{def.Bn.x1x2y1y2}
\ee
The ordinary generating function of these polynomials have a nice Jacobi-type continued fraction
\cite[Theorem~3.2]{Sokal-Zeng_22}.

The second polynomial is \cite[eq.~(3.11)]{Sokal-Zeng_22}
\begin{eqnarray}
   & &  \hspace*{-3mm}
		B_n^{(2)}(x_1,x_2,y_1,y_2,v_1,v_2,p_1,p_2,q_1,q_2,r)
        \nonumber \\[2mm]
   & & 
   \qquad=\;
   \sum_{\pi \in \Pi_n}
      x_1^{m_1(\pi)} x_2^{m_{\ge 2}(\pi)}
      y_1^{\erecin(\pi)} y_2^{\erecop(\pi)}
      v_1^{\nerecin(\pi)} v_2^{\nerecop(\pi)}
        \:\times
        \qquad\qquad
        \nonumber \\[1mm]
    & & \qquad\qquad\quad\;\:
      p_1^{\crin(\pi)} p_2^{\crop(\pi)}
      q_1^{\nein(\pi)} q_2^{\neop(\pi)}
      r^{\psne(\pi)}
   \:.
   \qquad
 \label{def.Bn.crossnest.refined}
\end{eqnarray}
The ordinary generating function of these polynomials have a nice Jacobi-type continued fraction as well
\cite[Theorem~3.3]{Sokal-Zeng_22}.

Our third polynomial is analogous to
\eqref{def.Bn.crossnest.refined}
but uses block-records, overlaps, coverings and pseudo-coverings
in place of exclusive records, crossings, nestings and pseudo-nestings:
\begin{eqnarray}
   & &  \hspace*{-3mm}
   B^{(3)}_n(x_1,x_2,y_1,y_2,v_1,v_2,p_1,p_2,q_1,q_2,r)
        \nonumber \\[2mm]
   & & 
   \qquad=\;
   \sum_{\pi \in \Pi_n}
      x_1^{m_1(\pi)} x_2^{m_{\ge 2}(\pi)}
      y_1^{\brecin(\pi)} y_2^{\brecop(\pi)}
      v_1^{\nbrecin(\pi)} v_2^{\nbrecop(\pi)}
        \:\times
        \qquad\qquad
        \nonumber \\[-1mm]
    & & \qquad\qquad\quad\;\:
      p_1^{\ovin(\pi)} p_2^{\ov(\pi)}
      q_1^{\covin(\pi)} q_2^{\cov(\pi)}
      r^{\pscov(\pi)}
   \:.
   \qquad
 \label{def.Bn.ovcov.brec}
\end{eqnarray}
It turns out that these polynomials are not merely {\em analogous}\/
to \reff{def.Bn.crossnest.refined};
they are {\em identical}\/ to \reff{def.Bn.crossnest.refined} 
\cite[Theorem~3.5]{Sokal-Zeng_22}.
See historical remark in \cite[Section~3.4]{Sokal-Zeng_22}.
	
In this paper, we refrain from working with Sokal and Zeng's master continued fractions for set partitions
which involve infinitely many statistics on set partitions \cite[Sections~3.7-3.10]{Sokal-Zeng_22}.
However, it is easy to generalise our results to the infinite-variable setting.

\subsection{Instances of the cyclic sieving phenomenon}
\label{sec.csp.setpart}

%
%
%
%
%

We are now ready to state instances of the cyclic sieving phenomena on set partitions 
with respect to the Kasraoui--Zeng involution
or the CDDSY involution,
or as a matter of fact, 
for any involution having $F_{2n-1}$ fixed points.

We first state a very general result in Proposition~\ref{prop.setpart.csp}.

\begin{proposition}
Let $\Stat{} : \Pi_n \to \mathbb{N}$ be a statistic defined on set partitions.

\begin{itemize}
	\item[(a)] Assume that $\Stat{}$ is obtained as a specialisation of the polynomial $B_n^{(1)}(x_1,x_2,y_1,y_2,v_1,v_2)$ 
	defined in~\eqref{def.Bn.x1x2y1y2}, 
	i.e.,
		there exists polynomials $x_1(q), x_2(q),y_1(q),y_2(q),v_1(q),v_2(q) \in\mathbb{R}[q]$
	satisfying
	\be
	\sum_{\pi\in \Pi_n} q^{\Stat{}} 
		\;=\;
	B_n^{(1)}(x_1(q),x_2(q),y_1(q),y_2(q),v_1(q),v_2(q))
	\ee
	for all $n\geq 0$.
	Further assume that the following equations are satisfied
	\begin{subeqnarray}
		x_1(-1)
		&=&
		y_1(-1)
		\;=\;
		1
		\\
		x_2(-1)\cdot y_2(-1)
		&=&
		1
		\\
		 x_2(-1)\cdot v_2(-1)
                &=&
                -1
	\label{eq.csp.cond.setpart}
	\end{subeqnarray}
	Then the statistic $\Stat{}$ exhibits the cyclic sieving phenomenon with respect to involutions having $F_{2n-1}$ fixed points \cite[A001519]{OEIS}.

	\item[(b)]
	Assume that $\Stat{}$ is obtained as a specialisation of the polynomial\\ 
	$B_n^{(2)}(x_1,x_2,y_1,y_2,v_1,v_2,p_1,p_2,q_1,q_2,r)$
	defined in~\eqref{def.Bn.crossnest.refined},
        i.e.,
	there exists polynomials $x_1(q), x_2(q),y_1(q),y_2(q),v_1(q),v_2(q),p_1(q),p_2(q),q_1(q),q_2(q),r(q) \in\mathbb{R}[q]$
        satisfying
        \be
        \sum_{\pi\in \Pi_n} q^{\Stat{}}
                \;=\;
		B_n^{(2)}(x_1(q),x_2(q),y_1(q),y_2(q),v_1(q),v_2(q), p_1(q),p_2(q),q_1(q),q_2(q),r(q))
        \ee
        for all $n\geq 0$.
        Further assume that the following equations are satisfied
        \begin{subeqnarray}
                x_1(-1)
                &=&
                1
                \\
                x_2(-1)\cdot y_2(-1)
                &=&
                1
                \\
                 r(-1)+y_1(-1)
                &=&
                2
		\\
		p_2(-1) + q_2(-1)\cdot x_2(-1)\cdot v_2(-1)
		&=&
		0
	\label{eq.csp.cond.pqgen.setpart}
        \end{subeqnarray}
        Then the statistic $\Stat{}$ exhibits the cyclic sieving phenomenon with respect to involutions having $F_{2n-1}$ fixed points.

	\item[(c)]
        Assume that $\Stat{}$ is obtained as a specialisation of the polynomial\\
	$B_n^{(3)}(x_1,x_2,y_1,y_2,v_1,v_2,p_1,p_2,q_1,q_2,r)$
        defined in~\eqref{def.Bn.ovcov.brec}.
	Further assume that equations~\eqref{eq.csp.cond.pqgen.setpart}
	are satisfied.
        Then the statistic $\Stat{}$ exhibits the cyclic sieving phenomenon with respect to involutions having $F_{2n-1}$ fixed points.

\end{itemize}
\label{prop.setpart.csp}
\end{proposition}

\begin{proof}
(a) We replace the variables in \cite[Theorem~3.2]{Sokal-Zeng_22} with  polynomials
$x_1(q)$, $x_2(q)$, $y_1(q)$, $y_2(q)$, $v_1(q)$, $v_2(q)$ 
and then set $q=-1$.
Using the conditions in~\eqref{eq.csp.cond.setpart}, 
it follows that the continued fraction~\cite[eq.~(3.7)]{Sokal-Zeng_22}
under this substitution becomes
\be
\sum_{n=0}^\infty \left. B_n^{(1)}(x_1(q), x_2(q),y_1(q),y_2(q),v_1(q),v_2(q)) \right|_{q=-1} t^n
\;=\;
\dfrac{1}{1-t - \dfrac{t^2}{1- 2 t}}\;.
\ee
The right-hand side is the generating function for the odd Fibonacci numbers $F_{2n-1}$.

\medskip

The proof of (b) and (c) is similar to that of (a)
but now we use \cite[Theorems~3.3 and~3.5]{Sokal-Zeng_22}.
\end{proof}

As a consequence of Proposition~\ref{prop.setpart.csp},
we can state our main theorem for set partitions as an easy corollary:

\begin{theorem}[CSP for set partitions]
The following statistics on set partitions exhibit the cyclic sieving phenomenon
with respect to involutions having $F_{2n-1}$ fixed points \cite[A001519]{OEIS},
(in particular, with respect to the Kasraoui--Zeng and the CDDSY involutions):
\begin{quote}

non-exclusive records ($=n - |\pi| - {\rm erec}(\pi)$ ), non-exclusive record openers,

crossings, nestings, crossing of opener type, nesting of opener type

non-block records (insiders and openers), openers that are not block-records, overlaps, coverings

{\color{red}


}
\end{quote}

\end{theorem}

\subsection{A remark on Mahonian statistics on set partitions}
\label{sec.mahonian}

%

In \cite{Milne_82}, Milne introduced the inversion statistic
on set partitions;
following Wachs and White \cite{Wachs_91}
we write it as  \cite[eq.~(3.61a)]{Sokal-Zeng_22}
\be
{\rm lb}(\pi)
\;\eqdef\;
\#\{(B_1,B_2,k)
\:\colon \;
\min B_1< \min B_2 < k\in B_1
\}\;.
\ee
The following continued fraction identity was 
proved in \cite{Zeng_94}
\be
\sum_{n=0}^\infty \sum_{\pi \in \Pi_n} q^{{\rm lb}(\pi)} t^n
\;=\;
\cfrac{1}{1 - \cfrac{ t}{ 1- \cfrac{t}{1- \cfrac{q t}{ 1- \cfrac{(1+q) t}{1- \cfrac{q^2  t}{1- \cfrac{(1+q+q^2)t}{1-\ldots})}}}}}}\;.
\label{eq.setpart.inv.cf}
\ee
Various other statistics have been shown to be equidistributed with
the inversion statistic ${\rm lb}(\pi)$,
see more details in 
\cite[Sections~3.11,~3.12]{Sokal-Zeng_22}
and in
\cite{Liu_22};
these are known as the {\em Mahonian statistics on set partitions}.

Notice that on setting 
$q=-1$ in the continued fraction~\eqref{eq.setpart.inv.cf},
we get
\be
\sum_{n=0}^\infty \sum_{\pi \in \Pi_n} (-1)^{{\rm lb}(\pi)} t^n
\;=\;
\dfrac{1}{1- t- t^2}
\ee
where clearly the right-hand side is the generating function of the Fibonacci numbers $F_n$.
%
Thus, an involution $\phi: \Pi_n \to \Pi_n$
must have Fibonacci many fixed points
if any Mahonian statistic were to exhibit the cyclic-sieving phenomenon with respect to $\phi$.

We will now construct such an involution by 
using the Kasraoui--Zeng bijection $\SZbij_{n} : \Pi_{n} \to \motzkinset_{n}$.
For a level step of a Motzkin path $\omega \in \motzkinset_{n}$
we rewrite some labels as $c_0, c_1,\ldots, c_h$
where  
\be
c_i 
\;=\;
\begin{cases}
(2,0) &\quad \text{if $i = h$}\\
(1,i) &\quad \text{otherwise}
\end{cases}
\;.
\ee
Now let $\bar{\phi} : \motzkinset_n \to \motzkinset_n$
be the following involution on Charlier diagrams
that sends $(\omega, \xi)$ to $(\omega, \xihat)$
where $\xihat = (\xihat_1, \ldots, \xihat_n)$ 
is given by
\be
\xihat_i
\;=\;
\begin{cases}
        \xi_i = 0 &\qquad \text{if step $i$ is a rise}\\
         h-1 - \xi_i & \qquad \text{if step $i$ is a fall}\\
         c_{h-k} &\qquad \text{if step $i$ is a level step and $\xi_i = c_k$}
\end{cases}
\;.
\ee
It is clear that $\bar{\phi}$
is an involution on $\motzkinset_n$
with $F_n$ fixed points.
The composition of maps
$\phi\coloneqq\left(\SZbij_n\right)^{-1}\circ \bar{\phi}\circ \SZbij_n$
gives us an involution on set partitions 
with the same number of fixed points.
We leave out all the details and state the following result.

\begin{theorem}
The Mahonian statistics on set partitions 
exhibit the cyclic sieving phenomenon 
with respect to involutions have
$F_n$ fixed points
(in particular the involution
$\left(\SZbij_n\right)^{-1}\circ \bar{\phi}\circ \SZbij_n$).
\end{theorem}

\section{Perfect matchings}
\label{sec.perfect}

A {\em perfect matching of $[2n]$} is an involution on the set $[2n]$ with no fixed points;
we use $\mathcal{I}_{2n}\subseteq \mathfrak{S}_{2n}$ to denote the set of perfect matchings.
It is known that the cardinality of this set is $|\mathcal{I}_{2n}| = (2n-1)!!$.
An important property of the sequence of numbers $(2n-1)!!$ is that they have the following Stieltjes-type continued fraction
\cite[sec.~29]{Euler_1760}
\be
\sum_{n=0}^\infty (2n-1)!!\, t^n
\;=\;
\dfrac{1}{1- \dfrac{t}{1-\dfrac{2t}{1-\dfrac{3t}{1-\ldots}}}}
\;.
\ee
We will use various generalisations of this identity in this section.

In Section~\ref{sec.perfect.stats} we define some statistics on set partitions.
Then in Section~\ref{sec.KZ.perfect} we recall the Kasraoui--Zeng
involution \cite{Kasraoui_06} on perfect matchings
in terms of labelled Dyck paths and count its number of fixed points.
After this, in Section~\ref{sec.cddsy.perfect}, 
we look at Chen et al's involution \cite{Chen_07} 
on perfect matchings
and count its number of fixed points.
Then in Section~\ref{sec.SZpoly.perfect},
we will use the statistics defined in Section~\ref{sec.perfect.stats}
to recall multivariate generating polynomials
introduced by Sokal and Zeng \cite{Sokal-Zeng_22}.
Finally, we state our main results 
in Section~\ref{sec.csp.perfect}.

\subsection{Statistics on perfect matchings}
\label{sec.perfect.stats}

Let $\sigma\in \mathcal{I}_{2n}$ be a perfect matching on $[2n]$.
Each $\sigma$ has $n$ cycles of size $2$,
and each cycle has a cycle peak and a cycle valley.
Following Sokal and Zeng \cite{Sokal-Zeng_22}, 
we classify these according to their {\em parity} and their {\em record classification}
(by considering them as permutations of $[2n]$).
We then obtain the following eight statistics,
the following four for cycle peaks:
even cycle-peak antirecord (ecpar),
odd cycle-peak antirecord (ocpar),
even cycle-peak non-antirecord (ecpnar),
odd cycle-peak non-antirecord (ocpnar),
and the remaining four for cycle valleys:
even cycle-valley record (ecvr),
odd cycle-valley record (ocvr),
even cycle-valley non-record (ecvnr),
odd cycle-valley non-record (ocvnr).

After this, we introduce crossings and nestings on perfect matchings.
We think of $\sigma$ as a permutation and notice that for involutions
we always have $\ucross(\sigma) = \lcross(\sigma)$
and $\unest(\sigma) = \lnest(\sigma)$;
so we denote these quantities simply as $\cros(\sigma)$ and $\nes(\sigma)$, respectively.

We then further refine crossings and nestings as follows:
we say that a quadruplet $i<j<k<l$ 
is an {\em even} (resp. {\em odd}) crossing
if the index $j$ is even (resp. odd) and 
the quadruplet forms a (upper) crossing;
we use $\ecr(\sigma)$, and $\ocr(\sigma)$ to denote the number of even and odd crossings, respectively.
Similarly, we define even and odd nestings according to the 
the parity of the second index $j$ for quadruplets which form upper nestings;
we use $\ene(\sigma)$, and $\oone(\sigma)$ to denote the number of even and odd nestings, respectively.

\subsection{The Kasraoui--Zeng involution on perfect matchings}
\label{sec.KZ.perfect}

The Kasraoui--Zeng involution on set partitions can be restricted down to perfect matchings.
This restriction to perfect matchings exchanges the
number of crossings and nestings while keeping various other statistics fixed.
We now state it this involution via labelled Dyck paths.

Consider the set of $(\bfscra, \bfscrb)$-labelled  Dyck paths of length $2n$,
where
the labels $\xi_i$ lie in the sets
\begin{subeqnarray}
   \scra_h        & = &  \{0 \}  \qquad\quad\;\;\hbox{for $h \ge 0$}  \\
   \scrb_h        & = &  \{0,\ldots,  h-1 \}    \quad\hbox{for $h \ge 1$}  
 \label{def.abc.perfectm}
\end{subeqnarray}

Let $\dyckperfect_{n}$ denote this set of $(\bfscra, \bfscrb)$-labelled
Dyck paths of length $2n$ with label sets~(\ref{def.abc.perfectm}a,b).
The Kasraoui--Zeng bijection as per \cite{Sokal-Zeng_22}
is a bijective map $\SZbijperfect_{2n} : \mathcal{I}_{2n} \to \dyckperfect_{2n}$
(seen as the restriction of the bijection $\SZbij_{2n}:\Pi_{2n}\to \motzkinset_{2n}$ 
when there are no level steps).

Let $(\omega,\xi)\in \dyckperfect_{2n}$
where $\xi = (\xi_1,\ldots,\xi_{2n})$.
Let $\xicomplement = (\xicomplement_1,\ldots, \xicomplement_{n})$ be a sequence where
$\xicomplement_i$
is given by
\be
\xicomplement_i
\;=\;
\begin{cases}
        \xi_i = 0 &\qquad \text{if step $i$ is a rise}\\
         h-1 - \xi_i & \qquad \text{if step $i$ is a fall}\\
\end{cases}
\ee

It is clearly follows that $(\omega,\xicomplement)\in \dyckperfect_{2n}$
and that $(\xicomplement)^{{\rm c}} = \xi$.
Thus, the map $(\omega, \xi)\mapsto (\omega,\xicomplement)$
is an involution on the set $\dyckperfect_{2n}$;
we use $\phi_{2n}$ to denote this involution.
The Kasraoui--Zeng involution on perfect matchings
is the map $\bigl(\SZbijperfect_{2n}\bigr)^{-1} \circ \phi_{2n} \circ \SZbijperfect_{2n}$.

As in the set partition case,
we have for a perfect matching $\sigma$ with 
$\widehat{\sigma} \coloneqq \biggl(\bigl(\SZbijperfect_{2n}\bigr)^{-1} \circ \phi_{2n} \circ \SZbijperfect_{2n}\biggr)(\sigma)$,
we have
\be
(\cros(\sigma), \nes(\sigma))
\;=\;
(\nes(\widehat{\sigma}), \cros(\widehat{\sigma})) \;.
\ee

\bigskip

We now establish the number of fixed points of the Kasraoui--Zeng involution on perfect matchings.

\begin{lemma}
	The number of fixed points of $\mathcal{I}_{2n}$ under the action of the Kasraoui--Zeng involution
$\bigl(\SZbijperfect_{2n}\bigr)^{-1} \circ \phi_{2n} \circ \SZbijperfect_{2n}$
is $1$
and thus is generated by the continued fraction
\be
\cfrac{1}{1 - t }\;.
\ee
\label{lem.csp.fixed.KZ.perfect}
\end{lemma}

The proof of this lemma is similar to that of Lemma~\ref{lem.csp.fixed.KZ} and hence, we leave it to the reader.



\subsection{The Chen--Deng--Du--Stanley--Yan (CDDSY) involution on perfect matchings}
\label{sec.cddsy.perfect}

In  \cite{Chen_07}, Chen et al showed that the bijection
$\CDDSY_n$ can be restricted to perfect matchings.
Now this bijection is to a subclass of vacillating tableaux,
called {\em oscillating tableaux},
which were studied by Sundaram in her PhD thesis \cite{Sundaram_86}.

We now recall the definition of oscillating tableaux.
\begin{definition}
An {\em oscillating tableaux} $T_{\lambda}^{n}$ of shape $\lambda$ 
and length $n$ is a sequence
$\emptyset=\lambda^0,\lambda^1,\lambda^2,\ldots, \lambda^{n} = \lambda$ of integer partitions
with the property that $\lambda^{i}$ is obtained 
either deleting or adding a square to $\lambda^{i-1}$.
\end{definition}

In terms of the $U, D$ and $I$ operators on the Young's lattice $\mathbb{Y}$,
we can restate this definition as follows:
given a $\lambda^{i-1}$ in the sequence of a vacillating tableaux,
the integer partition $\lambda^{i}$ is obtained by applying the operator 
$U + D$ to $\lambda^{i-1}$ and then picking one of the summands.

Let $\mathcal{OT}_{n}$ denote the set of all oscillating tableaux of shape $\emptyset$
and length $n$. 
The CDDSY bijection is a bijective map
$\CDDSY_{2n} : \mathcal{I}_{2n} \to \OT_{2n}$
(seen as the restriction of the bijection $\CDDSY_{2n}:\Pi_{2n}\to \VT_{2n}$ 
when the partition $\pi$ has no insiders or singletons and
all consecutive repeatitions are removed from the vacillating tableau).
Let $\tau_n : \OT_n$ be the involution on oscillating tableaux 
that is obtained by sending each integer partition to its conjugate.
Then the CDDSY involution for perfect matchings 
is the map
$\bigl(\CDDSY_{2n}\bigr)^{-1} \circ \tau_{2n} \circ \CDDSY_{2n}$.

We now show that the number of fixed points of the CDDSY involution
is also given by Lemma~\ref{lem.csp.fixed.KZ.perfect}.

\begin{lemma}
The number of fixed points of $\mathcal{I}_{2n}$ under the action of the 
CDDSY involution
$\bigl(\CDDSY_{2n}\bigr)^{-1} \circ \tau_{2n} \circ \CDDSY_{2n}$
is $1$
and thus is generated by the continued fraction
\be
\cfrac{1}{1 - t }\;.
\ee
\label{lem.csp.fixed.CDDSY.perfect}
\end{lemma}

\begin{proof}
It suffices to count the number of fixed points of 
the involution $\tau_{2n}$.
Notice that an oscillating tableau 
$(\lambda^0,\ldots, \lambda^{2n})\in \OT_{2n}$
is a fixed point of $\tau_{2n}$ if and only if 
$\lambda^{2i} = \emptyset$ and $\lambda^{2i-1} = 1$.
This finishes the proof.
\end{proof}

\subsection{Sokal and Zeng's multivariate polynomials enumerating statistics on perfect matchings}
\label{sec.SZpoly.perfect}

In this section, we recall two of Sokal and Zeng's multivariate polynomials enumerating statistics on perfect matchings;
first we shall state \cite[eq.~(4.4)]{Sokal-Zeng_22} which does not have any $p,q$-statistics,
then we shall state their $p,q$-generalisation involving crossings and nestings
\cite[eq.~(4.22)]{Sokal-Zeng_22}.

The first polynomial is \cite[eq.~(4.4)]{Sokal-Zeng_22}
\begin{subeqnarray}
   M_n(x,y,u,v)
   & = &
   \sum_{\sigma \in \scri_{2n}}
      x^{\ecpar(\sigma)} y^{\ocpar(\sigma)}
         u^{\ecpnar(\sigma)}  v^{\ocpnar(\sigma)}
                        \slabel{eq.matching.fourvar.a}  \\[2mm]
   & = &
   \sum_{\sigma \in \scri_{2n}}
      x^{\ocvr(\sigma)} y^{\ecvr(\sigma)}
         u^{\ocvnr(\sigma)}  v^{\ecvnr(\sigma)}
                        \slabel{eq.matching.fourvar.b}
   \;.
 \label{eq.matching.fourvar}
\end{subeqnarray}
The equaltity between the two descriptions follows from applying reversal followed by complementation to $\sigma$.
The ordinary generating function of these polynomials have a nice Stieltjes-type continued fraction
\cite[eq.~(4.2), Theorem~4.1]{Sokal-Zeng_22}.

The second polynomial is \cite[eq.~(4.22)]{Sokal-Zeng_22}
\begin{subeqnarray}
   & &
   \!\!
   M_n(x,y,u,v,p_+,p_-,q_+,q_-)
         \nonumber \\[2mm]
   & & \qquad =\;
   \sum_{\sigma \in \scri_{2n}}
      x^{\ecpar(\sigma)} y^{\ocpar(\sigma)}
         u^{\ecpnar(\sigma)}  v^{\ocpnar(\sigma)}
         p_+^{\ocr(\sigma)} p_-^{\ecr(\sigma)}
         q_+^{\oone(\sigma)} q_-^{\ene(\sigma)}
     \qquad\qquad
      \\[2mm]
   & & \qquad =\;
   \sum_{\sigma \in \scri_{2n}}
      x^{\ocvr(\sigma)} y^{\ecvr(\sigma)}
         u^{\ocvnr(\sigma)}  v^{\ecvnr(\sigma)}
         p_+^{\ecr(\sigma)} p_-^{\ocr(\sigma)}
         q_+^{\ene(\sigma)} q_-^{\oone(\sigma)}
   \;,
 \slabel{def.matching.fourvar.pq.plusminus.b}
 \label{def.matching.fourvar.pq.plusminus}
\end{subeqnarray}
The equaltity between the two descriptions again follows from applying reversal followed by complementation to $\sigma$.
The ordinary generating function of these polynomials have a nice Stieltjes-type continued fraction
\cite[Theorem~4.4]{Sokal-Zeng_22}.

\subsection{Instances of the cyclic sieving phenomenon}
\label{sec.csp.perfect}

We are now ready to state instances of the cyclic sieving phenomena on perfect matchings
with respect to the Kasraoui--Zeng involution on perfect matchings
or the CDDSY involution on perfect matchings,
or as a matter of fact, 
for any involution having $1$ fixed point for every $n\geq 0$.

We first state a very general result in Proposition~\ref{prop.perfect.csp}.

\begin{proposition}
Let $\Stat{} : \mathcal{I}_{2n} \to \mathbb{N}$ be a statistic defined on perfect matchings.

\begin{itemize}
	\item[(a)] Assume that $\Stat{}$ is obtained as a specialisation of the polynomial $M_n(x,y,u,v)$ 
	defined in~\eqref{eq.matching.fourvar},
	i.e.,
	there exists polynomials $x(q), y(q),u(q),v(q) \in\mathbb{R}[q]$
	satisfying
	\be
	\sum_{\sigma\in \mathcal{I}_{2n}} q^{\Stat{}} 
		\;=\;
	M_n(x(q),y(q),u(q),v(q))
	\ee
	for all $n\geq 0$.
	Further assume that the following equations are satisfied
	\begin{subeqnarray}
		x(-1)
		&=&
		1
		\\
		y(-1) \,+\, v(-1)
                &=&
                0
	\label{eq.csp.cond.perfect}
	\end{subeqnarray}
	Then the statistic $\Stat{}$ exhibits the cyclic sieving phenomenon with respect to involutions 
	having $1$ fixed point for every $n\geq 0$.

	\item[(b)]
	Assume that $\Stat{}$ is obtained as a specialisation of the polynomial\\
	$M_n(x,y,u,v,p_+,p_-,q_+,q_-)$	
	defined in~\eqref{def.matching.fourvar.pq.plusminus},
        i.e.,
	there exists polynomials $x(q)$, $y(q)$, $u(q)$, $v(q)$, $p_+(q)$, $p_-(q)$, $q_+(q)$, $q_-(q) \in\mathbb{R}[q]$
        satisfying
        \be
        \sum_{\sigma\in \mathcal{I}_{2n}} q^{\Stat{}}
                \;=\;
		M_n(x(q),y(q),u(q),v(q),p_+(q),p_-(q),q_+(q),q_-(q))
        \ee
        for all $n\geq 0$.
        Further assume that the following equations are satisfied
        \begin{subeqnarray}
                x(-1)
                &=&
                1
                \\
		p_+(-1)\cdot y(-1) \,+\, q_{+}(-1)\cdot v(-1)
                &=&
                0
	\label{eq.csp.cond.pqgen.perfect}
        \end{subeqnarray}
        Then the statistic $\Stat{}$ exhibits the cyclic sieving phenomenon with respect to involutions having
	$1$ fixed point for every $n\geq 0$.

\end{itemize}
\label{prop.perfect.csp}
\end{proposition}

\begin{proof}
(a) We replace the variables in \cite[Theorem~4.1]{Sokal-Zeng_22} with  polynomials
$x_1(q)$, $x_2(q)$, $y_1(q)$, $y_2(q)$, $v_1(q)$, $v_2(q)$ 
and then set $q=-1$.
Using the conditions in~\eqref{eq.csp.cond.perfect}, 
it follows that the continued fraction~\cite[eq.~(4.2)]{Sokal-Zeng_22}
under this substitution becomes
\be
\sum_{n=0}^\infty \left. M_n(x(q),y(q),u(q),v(q)) \right|_{q=-1} t^n
\;=\;
\dfrac{1}{1-t }\;.
\ee

\medskip

The proof of (b) but now we use \cite[Theorem~4.4]{Sokal-Zeng_22}.
\end{proof}

As a consequence of Proposition~\ref{prop.perfect.csp},
we can state our main theorem for perfect matchings as an easy corollary:

\begin{theorem}[CSP for perfect matchings]
The following statistics on perfect matchings exhibit the cyclic sieving phenomenon
with respect to involutions having $1$ fixed point,
(in particular, with respect to the Kasraoui--Zeng involution on perfect matchings
or the CDDSY involution on perfect matchings):
\begin{quote}

odd cycle peak antirecord, odd cycle peak non-antirecord, even cycle valley record, even cycle valley non-record,

crossing, nesting, odd crossing, odd nesting, even crossing, even nesting

\end{quote}

\end{theorem}

%

We conclude this section by a discussion on the {\em roof} statistic of Benaych-Georges, Cuenca and Gorin 
\cite[Section~4.1]{Benaych-Georges_22}
which exhibits the cyclic sieving phenomenon only when $n$ is even.

For a perfect matching $\sigma\in \mathcal{I}_{2n}$, we say that an index $i$ forms a non-crossing roof 
if $i$ is an opener and $\cros(i,\sigma) = 0$;
we use $\roof(\sigma)$ to denote the number of non-crossing roofs of $\sigma$
\cite[see footnote~27]{Sokal-Zeng_22}.
Thus, we can obtain the generating polynomial for $\roof(\sigma)$ as follows:
\be
\sum_{\sigma\in \mathcal{I}_{2n}} x^{\roof(\sigma)} 
\;=\;
M_n(x,x,1,1)
\;.
\label{eq.roof}
\ee
We immediately obtain the following result:

\begin{theorem}
The statistic $\roof(\sigma)$ exhibits the cyclic sieving phenomenon 
with respect to involutions having $1$ fixed point
(in particular, with respect to the Kasraoui--Zeng involution on perfect matchings
or the CDDSY involution on perfect matchings)
when $n$ is even.
\end{theorem}

\begin{proof}
By~\eqref{eq.roof},
we get that the generating function for this statistic can be expressed as a continued fraction 
\cite[Theorem~4.1]{Sokal-Zeng_22}:
\be
	\sum_{n=0}^\infty \sum_{\sigma\in \mathcal{I}_{2n}} x^{\roof(\sigma)}  t^n
	\;=\;
	\dfrac{1}{1-\dfrac{xt}{1-\dfrac{(x+1)t}{1-\dfrac{(x+2)t}{1 - \dfrac{(x+3)t}{1-\ldots}}}}}
	\;.
\ee
Plugging in $x=-1$ in this equation, the right-hand-side reduces to 
\be
\dfrac{1}{1+t}
\ee
which gives us the desired outcome.
\end{proof}

\section{D-permutations and its subfamilies}
\label{sec.dperm}

We begin by recalling the definitions of D-permutations and its subfamilies in Section~\ref{subsec.dperm.definitions}.
Then we introduce a new involution, {\em the Genocchi--Corteel involution}, in 
Section~\ref{sec.genocchi.corteel};
we also count its number of fixed points.
After this, in Section~\ref{subsec.dspoly}, 
we recall multivariate generating polynomials introduced by Deb and Sokal 
\cite{Deb-Sokal_genocchi}.
Finally, we state our results for D-permutations in Section~\ref{subsec.csp.dperm},
and its subfamilies (D-semiderangements and D-derangements)
in Section~\ref{subsec.csp.dperm.sub}.

\subsection{Genocchi, median Genocchi numbers, and D-permutations}
\label{subsec.dperm.definitions}

We follow the terminology and conventions in \cite{Deb-Sokal_genocchi} and in \cite{Deb_23}.

The Genocchi numbers \cite[A110501]{OEIS}
\be
   (g_n)_{n \ge 0}
   \;=\;
   1, 1, 3, 17, 155, 2073, 38227, 929569, 28820619, 1109652905,
   \ldots
\ee
are odd positive integers \cite{Lucas_1877,Barsky_81,Han_18}
\cite[pp.~217--218]{Foata_08}
defined by the exponential generating function
\be
   t \, \tan(t/2)
   \;=\;
   \sum_{n=0}^\infty g_n \, {t^{2n+2} \over (2n+2)!}
   \;.
\ee

The median Genocchi numbers (or Genocchi medians for short) 
\cite[A005439]{OEIS}
\begin{equation}
        (h_n)_{n \ge 0}
   \;=\;
   1, 1, 2, 8, 56, 608, 9440, 198272, 5410688, 186043904,
   \ldots
   \;.
\end{equation}
are defined by \cite[p.~63]{Han_99b}
\be
   h_n  \;=\;  \sum_{i=0}^{n-1} (-1)^i \, \binom{n}{2i+1} \, g_{n-1-i}
   \;.
 \label{eq.hn.binomgn}
\ee
See \cite[Sections~2.5,~2.6]{Deb-Sokal_genocchi} and references therein 
for continued fractions associated to the Genocchi and median Genocchi numbers.

The median Genocchi numbers enumerate a class of permutations called D-permutations (short for Dumont-like permutations),
they were introduced by Lazar and Wachs in \cite{Lazar_20,Lazar_22,Lazar_23}.
A permutation of $[2n]$ is called a D-permutation in case
$2k-1 \le \sigma(2k-1)$ and $2k \ge \sigma(2k)$ for all $k$,
i.e., it contains no even excedances and no odd anti-excedances.
Let us say also that a permutation is an
{\em e-semiderangement}\/ (resp.\ {\em o-semiderangement}\/)
in case it contains no even (resp.~odd) fixed points;
it is a {\em derangement}\/ in case it contains no fixed points at all.
A D-permutation that is also an
e-semiderangement (resp.\ o-semiderangement, derangement)
will be called a \textbfit{D-e-semiderangement}
(resp.\ \textbfit{D-o-semiderangement}, \textbfit{D-derangement}).
A D-permutation that contains exactly one cycle is called a \textbfit{D-cycle}. Notice that a D-cycle is also a D-derangement.
Let $\dperm_{2n}$
(resp.~$\dperm^{\rm e}_{2n}, \dperm^{\rm o}_{2n}, \dperm^{\rm eo}_{2n}, \dcycle_{2n}$)
denote the set of all D-permutations
(resp.\ D-e-semiderangements, D-o-semiderangements, D-derangements, D-cycles) of $[2n]$.
For instance,
\begin{subeqnarray}
   \dperm_2  & = & \{ 12 ,\, 21^{\rm eo} \}  \\[1mm]
   \dperm_4  & = & \{ 1234 ,\, 1243 ,\, 2134 ,\, 2143^{\rm eo} ,\,
                     3142^{\rm eo} ,\, 3241^{\rm o} ,\, 4132^{\rm e} ,\, 4231 \}\\[1mm]
\dcycle_2 & = & \{21\}\\[1mm]
\dcycle_4 & = & \{3142\}
\end{subeqnarray}
where $^{\rm e}$ denotes e-semiderangements that are not derangements,
$^{\rm o}$ denotes o-semi\-de\-range\-ments that are not derangements,
and $^{\rm eo}$ denotes derangements.

It is known \cite{Dumont_74,Dumont_94,Lazar_22,Lazar_20, Deb-Sokal_genocchi}
that
\begin{subeqnarray}
   |\dperm_{2n}|    & = &  h_{n+1}  \\[0.5mm]
   |\dperm^{\rm e}_{2n}|  \;=\; |\dperm^{\rm o}_{2n}|  & = &  g_n \\[1mm]
   |\dperm^{\rm eo}_{2n}|  & = &  h_n\\[0.5mm]
    |\dcycle_{2n}| & = & g_{n-1}
\end{subeqnarray}

\subsection{Deb--Sokal bijection and the Genocchi--Corteel involution}
\label{sec.genocchi.corteel}

Randrianarivony \cite{Randrianarivony_97}
introduced a bijection between D-o-semiderangements
and labelled Dyck paths 
which is in a similar essence to the Foata--Zeilberger bijection \cite{Foata_90};
he used this bijection
to obtain continued fractions
counting various statistics on D-o-semiderangements.
Motivated by Randrianarivony's bijection, Deb and Sokal \cite{Deb-Sokal_genocchi}
introduced two new bijections involving all D-permutations,
one of which extends Randrianarivony's bijection.

Both bijections of Deb and Sokal
(\cite[Sections~6.1-6.3]{Deb-Sokal_genocchi})
and 
\cite[Section~6.5]{Deb-Sokal_genocchi})
are correspondences between $\dperm_{2n}$ and the set of
$(\bfscra, \bfscrb,\bfscrc)$-labelled
0-Schr\"oder paths of length $2n$, where
the labels $\xi_i$ lie in the sets
\cite[eq.~(6.1)]{Deb-Sokal_genocchi}
\begin{subeqnarray}
   \scra_h        & = &  \{0,\ldots, \lceil h/2 \rceil \}  \qquad\quad\;\;\hbox{for $h \ge 0$}  \\
   \scrb_h        & = &  \{0,\ldots, \lceil (h-1)/2 \rceil \}    \quad\hbox{for $h \ge 1$}  \\
   \scrc_0        & = &  \{0\}       \\
   \scrc_h        & = &  \emptyset   \qquad\qquad\qquad\qquad\quad\hbox{for $h \ge 1$}
 \label{def.abc.dperm}
\end{subeqnarray}
Let $\schroderset_{2n}$ denote this set of $(\bfscra, \bfscrb,\bfscrc)$-labelled
0-Schr\"oder paths of length $2n$
with label sets~\eqref{def.abc.dperm}.

In this paper, we only focus on their first bijection which involves record classification
\cite[Sections~6.1-6.3]{Deb-Sokal_genocchi}.
Let $\DSbij_{2n} : \dperm_{2n} \to \schroderset_{2n}$
denote this bijection.

\bigskip

We now describe an involution which is the analog of the {\bf Corteel} involution in this setting;
we call it the {\bf Genocchi--Corteel} involution.

\noindent{\bf Description of the Genocchi--Corteel involution:}

Let $(\omega,\xi)\in \schroderset_{2n}$
where $\xi = (\xi_1,\ldots,\xi_{2n})$.
Let $\xicomplement = (\xicomplement_1,\ldots, \xicomplement_{2n})$ be a sequence where
$\xicomplement_i$
is given by
\be
\xicomplement_i
\;=\;
\begin{cases}
      \xi_i =  \lceil h/2 \rceil   &\qquad \text{if $i$ is even, step $i$  is a rise and $\xi_i = \lceil h/2 \rceil$}\\
      \lceil h/2 \rceil - 1 - \xi_i	 &\qquad \text{if $i$ is even and step $i$  is a rise and $\xi_i < \lceil h/2 \rceil $ }\\
      \lceil (h-1)/2 \rceil - \xi_i   & \qquad \text{if $i$ is even and step $i$ is a fall}\\
      \lceil h/2 \rceil - \xi_i   &\qquad \text{if $i$ is odd and step $i$ is a rise}\\
      \lceil (h-1)/2 \rceil -1- \xi_i &\qquad \text{if $i$ is odd, step $i$ is a fall 
      and $\xi_i < \lceil (h-1)/2 \rceil$}\\
      \xi_i =  \lceil (h-1)/2 \rceil   &\qquad \text{if $i$ is odd, step $i$ is a fall 
      and $\xi_i = \lceil (h-1)/2 \rceil$}\\
      \xi_i = 0 &\qquad \text{if step $i$ is a long level step at height $0$}
\end{cases}
\ee

It clearly follows that $(\omega,\xicomplement)\in \schroderset_{2n}$
and that $(\xicomplement)^{{\rm c}} = \xi$.
Thus, the map $(\omega, \xi)\mapsto (\omega,\xicomplement)$
is an involution on the set $\schroderset_{2n}$;
we use $\phi_{2n}$ to denote this involution.
We define the  {\bf Genocchi--Corteel involution on D-permutations}
to be the map 
\be
\phi^{{\rm GC}}_{2n}
\;\eqdef\;
\bigl(\DSbij_{2n}\bigr)^{-1} \circ \phi_{2n} \circ \DSbij_{2n}\;.
\ee
(We call it the Genocchi--Corteel involution as we think it is the analog of the Corteel involution for objects counted by the 
Genocchi and median--Genocchi numbers, we will expand on this in Proposition~\ref{prop.GC.exchange.dperm} 
and Corollary~\ref{cor.GC.subclasses}.)

We now state some properties of the involution $\phi^{{\rm GC}}_{2n}$;
in particular, we will show that the cycle classification is preserved under this involution but the 
crossings and nestings are exchanged.

\begin{proposition}
Let $\sigma\in \dperm_{2n}$ be a D-permutation on $2n$ letters.
The following holds true under the Genocchi--Corteel involution:
\begin{itemize}

	\item If $i$ is a cycle valley (cycle double rise resp.)
	of $\sigma$ then then $i$ is also a cycle valley (cycle double rise)
	of $\GCsigma$,
		and furthermore, $\unest(i,\sigma) = \ucross(i,\GCsigma)$.

	\item If $i$ is a cycle peak (cycle double fall resp.)
	of $\sigma$ then then $i$ is also a cycle peak (cycle double fall)
	of $\GCsigma$,
                and furthermore, $\lnest(i,\sigma) = \lcross(i,\GCsigma)$.



	\item If $i$ is an even fixed point (odd fixed point resp.)
	of $\sigma$ then $i$ is also an even fixed point (odd fixed point)
	of $\GCsigma$,
                and furthermore, $\psnest(i,\sigma) = \psnest(i,\GCsigma))$.

		
\end{itemize}
\label{prop.GC.exchange.dperm}
\end{proposition}

\begin{proof}
As the Schr\"oder path $\omega$ stays the same under the involution,
the sets $F,G,F',G'$ introduced in 
\cite[Section~6.3]{Deb-Sokal_genocchi},
which only depend on $\omega$ are the same for both $\sigma$ as well as $\GCsigma$.
This along with the fact that $\xi_i=\xicomplement_i$ whenever
$i$ corresponds to (even or odd) fixed point
shows that the cycle types are preserved under the involution $\phi^{{\rm GC}}_{2n}$.

The exchange of crossing and nestings follows from the definition of 
$\xicomplement$ and 
\cite[Lemmas~6.2 and~6.4]{Deb-Sokal_genocchi}.
\end{proof}

As a corollary of Proposition~\ref{prop.GC.exchange.dperm}, we immediately obtain the following:

\begin{corollary}
The Genocchi--Corteel $\phi^{{\rm GC}}_{2n}$  involution 
on the set of D-permutations $\dperm_{2n}$ can be restricted down to involutions on the subclasses 
of D-e-semiderangements 
\[\left.\phi^{{\rm GC}}_{2n}\right|_{\dperm_{2n}^{{\rm e}}} : \dperm_{2n}^{{\rm e}} \to \dperm_{2n}^{{\rm e}},\]
D-o-semiderangements 
\[\left.\phi^{{\rm GC}}_{2n}\right|_{\dperm_{2n}^{{\rm o}}} : \dperm_{2n}^{{\rm o}} \to \dperm_{2n}^{{\rm o}},\]
and D-derangements 
\[\left.\phi^{{\rm GC}}_{2n}\right|_{\dperm_{2n}^{{\rm eo}}} : \dperm_{2n}^{{\rm eo}} \to \dperm_{2n}^{{\rm eo}}.\] 
For each of these subclasses, their respective involutions preserve the cycle classification of indices,
and exchange crossings and nestings statistics 
as mentioned in Proposition~\ref{prop.GC.exchange.dperm}.
\label{cor.GC.subclasses}
\end{corollary}

\bigskip

\noindent {\bf Fixed points of the Genocchi--Corteel involution:}

%
Let $\GCdperm_{2n}$ denote the set of fixed points of the Genocchi--Corteel involution, i.e.,
\be
\GCdperm_{2n}
\;\eqdef\;
\{\sigma\in \dperm_{2n} \:\colon\: \GCsigma =\sigma \}\;.
\ee
Similarly, let 
$\GCdperm_{2n}^{{\rm e}}$,
$\GCdperm_{2n}^{{\rm o}}$,
and
$\GCdperm_{2n}^{{\rm eo}}$
denote the sets of fixed points of the Genocchi--Corteel involution restricted to the subclasses
$\dperm_{2n}^{{\rm e}}$,
$\dperm_{2n}^{{\rm o}}$,
and
$\dperm_{2n}^{{\rm eo}}$,
respectively.
We will now compute the cardinality of these sets. 

Define the following polynomial counting even and odd fixed points on the set $\GCdperm_{2n}$:
\be
\GCP_n(\we,\wo)
\;=\;
\sum_{\sigma\in\GCdperm_{2n}}
\we^{\evenfix(\sigma)}
\wo^{\oddfix(\sigma)}
\;.
\label{eq.def.fixedpoint.GC}
\ee

The ordinary generating functions of the polynomials $\GCP_n$ is a nice rational function:
\begin{theorem}[Ordinary generating function for fixed points of the Genocchi--Corteel involution]
The ordinary generating function of the polynomials~\eqref{eq.def.fixedpoint.GC}
is given by
\be
\sum_{n=0}^\infty
\GCP_n(\we,\wo)
\;=\;
\dfrac{1}{1- \we \wo t - \dfrac{t}{1-(1+\we)(1+\wo)t}}\;.
\label{eq.thm.fixedpoint.GC}
\ee
\label{thm.fixedpoint.GC}
\end{theorem}
\begin{proof}
We will enumerate labels $\xi$ such that $\xi = \xicomplement$
while assigning the following weights to the steps $s_i$ of $\omega$:
\be
	\wt(s_i)
	\;=\;
\begin{cases}
	\we\cdot \wo
	\quad \text{if $s_i$ is a long level step at height $0$}\\
	\we
	\qquad \text{if $s_i$ is a rise from height $1$ with label $\xi_i = 1$}\\
	\wo
	\qquad \text{if $s_i$ is a fall to height $1$ with label $\xi_i = 1$}\\
	1 
	\qquad \text{otherwise}
\end{cases}
\ee

Notice that if for some index $i$, step $i$ is a rise starting at height $3$
or a fall ending at height $3$,
we can never have $\xi_i = \xicomplement_i$.
Therefore, the fixed points of the involution $\phi_{2n}$
are exactly the pairs $(\omega, \xi)$
where the $0$-Schr\"oder path $\omega$ does not go above height $3$.
Using general theory of combinatorics of Thron-type continued fractions, 
we obtain the continued fraction~\reff{eq.thm.fixedpoint.GC}.
\end{proof}


\medskip

We will finally substitute
$(\we,\wo) = (1,1)$, 
$(\we,\wo) = (1,0)$, $(\we,\wo) = (0,1)$, and $(\we,\wo) = (0,0)$,
in Theorem~\ref{thm.fixedpoint.GC}
to obtain the number of fixed points under the action of the Genocchi–Corteel involution
on D-permutations and its subclasses; we state this in the following corollary.

\begin{corollary}
\begin{itemize}

\item[(a)] The cardinality of the set $\GCdperm_{2n}$ is given by the sequence \cite[A154626]{OEIS}
\be
	\left|\GCdperm_{2n} \right|
	\;=\;
	2^n F_{2n-1}\;.
\ee

\item[(b)] The cardinality of the sets $\GCdperm_{2n}^{{\rm e}}$ and $\GCdperm_{2n}^{{\rm o}}$
	are given by the sequence $ \cite[A133494]{OEIS}$
\begin{eqnarray}
	\left|\GCdperm_{2n}^{{\rm e}} \right|
	\;=\;
	\left|\GCdperm_{2n}^{{\rm o}} \right|
	\;=\;
	\begin{cases}
	1 \qquad \text{for $n=0$}\\
	3^{n-1} \quad \text{for $n\geq 1$}
	\;.
	\end{cases}
\end{eqnarray}

\item[(c)] The cardinality of the set $\GCdperm_{2n}^{{\rm eo}}$ is given by the sequence \cite[A011782]{OEIS}
\be
	\left|\GCdperm_{2n}^{{\rm eo}} \right|
	\;=\;
	\begin{cases}
        1 \qquad \text{for $n=0$}\\
	2^{n-1}
	\quad \text{for $n\geq 1$}
	\end{cases}
	\;.
\ee

\end{itemize}
\label{lem.genocchi.corteel.fixed.points}
\end{corollary}

\subsection{Deb and Sokal's multivariate polynomials enumerating statistics on D-permutations and its subclasses}
\label{subsec.dspoly}

In this section, we recall two of Deb and Sokal's multivariate polynomials enumerating
statistics on  D-permutations and its subclasses;
first we shall state first polynomial \cite[eq.~(3.2)]{Deb-Sokal_genocchi}
then we shall state their first $p,q$-generalisation involving crossings and nestings \cite[eq.~(3.21)]{Deb-Sokal_genocchi}.
As in Section~\ref{sec.cf.permutations},
we will use the glossary of permutation statistics 
as stated in \cite[section~2.5]{Deb_thesis};
this includes the record-and-cycle classification 
\cite[section~2.5.1]{Deb_thesis},
and crossings~and~nestings
\cite[section~2.5.2]{Deb_thesis}.

We now introduce a polynomial in 12 variables that enumerates D-permutations
according to the parity-refined record-and-cycle classification:
\begin{eqnarray}
   & &
   P_n^{(1)}(x_1,x_2,y_1,y_2,u_1,u_2,v_1,v_2,\we,\wo,\ze,\zo)
   \;=\;
       \nonumber \\[4mm]
   & & \qquad\qquad
   \sum_{\sigma \in \dperm_{2n}}
   x_1^{\eareccpeak(\sigma)} x_2^{\eareccdfall(\sigma)}
   y_1^{\ereccval(\sigma)} y_2^{\ereccdrise(\sigma)}
   \:\times
       \qquad\qquad
       \nonumber \\[-1mm]
   & & \qquad\qquad\qquad\:
   u_1^{\nrcpeak(\sigma)} u_2^{\nrcdfall(\sigma)}
   v_1^{\nrcval(\sigma)} v_2^{\nrcdrise(\sigma)}
   \:\times
       \qquad\qquad
       \nonumber \\[3mm]
   & & \qquad\qquad\qquad\:
   \we^{\evennrfix(\sigma)} \wo^{\oddnrfix(\sigma)}
   \ze^{\evenrar(\sigma)} \zo^{\oddrar(\sigma)}
   \;.
 \label{def.Pn}
\end{eqnarray}
The polynomials \reff{def.Pn} have a beautiful T-fraction \cite[Theorem~3.3]{Deb-Sokal_genocchi}.

Our second polynomial is a polynomial in 22 variables
that generalizes \reff{def.Pn} by including
four pairs of $(p,q)$-variables
corresponding to the four refined types of crossings and nestings
as well as two variables corresponding to pseudo-nestings of fixed points:
\begin{eqnarray}
   & &
   \hspace*{-14mm}
   P_n^{(2)}(x_1,x_2,y_1,y_2,u_1,u_2,v_1,v_2,\we,\wo,\ze,\zo,p_{-1},p_{-2},p_{+1},p_{+2},q_{-1},q_{-2},q_{+1},q_{+2},\se,\so)
   \;=\;
   \hspace*{-1cm}
       \nonumber \\[4mm]
   & & \qquad\qquad
   \sum_{\sigma \in \dperm_{2n}}
   x_1^{\eareccpeak(\sigma)} x_2^{\eareccdfall(\sigma)}
   y_1^{\ereccval(\sigma)} y_2^{\ereccdrise(\sigma)}
   \:\times
       \qquad\qquad
       \nonumber \\[-1mm]
   & & \qquad\qquad\qquad\:
   u_1^{\nrcpeak(\sigma)} u_2^{\nrcdfall(\sigma)}
   v_1^{\nrcval(\sigma)} v_2^{\nrcdrise(\sigma)}
   \:\times
       \qquad\qquad
       \nonumber \\[3mm]
   & & \qquad\qquad\qquad\:
   \we^{\evennrfix(\sigma)} \wo^{\oddnrfix(\sigma)}
   \ze^{\evenrar(\sigma)} \zo^{\oddrar(\sigma)}
   \:\times
       \qquad\qquad
       \nonumber \\[3mm]
   & & \qquad\qquad\qquad\:
   p_{-1}^{\lcrosscpeak(\sigma)}
   p_{-2}^{\lcrosscdfall(\sigma)}
   p_{+1}^{\ucrosscval(\sigma)}
   p_{+2}^{\ucrosscdrise(\sigma)}
          \:\times
       \qquad\qquad
       \nonumber \\[3mm]
   & & \qquad\qquad\qquad\:
   q_{-1}^{\lnestcpeak(\sigma)}
   q_{-2}^{\lnestcdfall(\sigma)}
   q_{+1}^{\unestcval(\sigma)}
   q_{+2}^{\unestcdrise(\sigma)}
          \:\times
       \qquad\qquad
       \nonumber \\[3mm]
   & & \qquad\qquad\qquad\:
   \se^{\epsnest(\sigma)}
   \so^{\opsnest(\sigma)}
 \label{def.Pn.pq}
\end{eqnarray}
where
\be
   \epsnest(\sigma)
   \;=\;
   \sum_{i \in \Evenfix} \psnest(i,\sigma)
   \;,\qquad
   \opsnest(\sigma)
   \;=\;
   \sum_{i \in \Oddfix} \psnest(i,\sigma)
   \;.
\ee
These polynomials~\eqref{def.Pn.pq} also have a nice continued fraction
\cite[Theorem~3.9]{Deb-Sokal_genocchi}.

Deb and Sokal have also introduced several other multivariate polynomials in their paper. We, however, refrain from working with those.

\subsection{Instances of the cyclic sieving phenomenon on D-permutations}
\label{subsec.csp.dperm}

We are now ready to state instances of the cyclic sieving phenomena on D-permutations and its subfamilies 
with respect to the Genocchi--Corteel involution.

We first state a very general result in Proposition~\ref{prop.dperm.csp}.

\begin{proposition}
Let $\Stat{} : \dperm_{2n} \to \mathbb{N}$ be a statistic defined on D-permutations.

\begin{itemize}
	\item[(a)] Assume that $\Stat{}$ is obtained as a specialisation of the polynomial $P_n^{(1)}$ 
	defined in~\eqref{def.Pn},
	i.e., as a specialisation of the associated variables.
	Further assume that the following equations are satisfied
	\begin{subeqnarray}
		\ze(-1) \;=\; \zo(-1) \;=\; \we(-1) &=& \wo(-1) \;=\; 1\\
		x_1(-1)\cdot y_1(-1)
		&=&
		1
		\\
		\biggl[x_2(-1)\,+\, 1\biggr] \,\times\, \biggl[y_2(-1) + 1\biggr]
		&=&
		4
		\\
		\biggl[x_1(-1) \,+\, u_1(-1) \biggr] \,\times\, \biggl[y_1(-1) \,+\, v_1(-1)\biggr]
                &=&
                0
	\label{eq.csp.cond.dperm}
	\end{subeqnarray}
	Then the statistic $\Stat{}$ exhibits the cyclic sieving phenomenon with respect to involutions having $2^n F_{2n-1}$ fixed points.

	\item[(b)]
	Assume that $\Stat{}$ is obtained as a specialisation of the polynomial $P_n^{(2)}$ 
	defined in~\eqref{def.Pn.pq},
        i.e., as a specialisation of the associated variables.
	Further assume that the following equations are satisfied
	\begin{subeqnarray}
		\hspace*{-10mm}
		\ze(-1) \;=\; \zo(-1) \;=\; \we(-1) \;=\; \wo(-1) \;=\; \se(-1) &=& \so(-1) \;=\;
		1\\
                x_1(-1)\,\times\, y_1(-1)
                &=&
                1
                \\
		\biggl[ x_2(-1) \,+\, 1 \biggr] \,\times\, \biggl[ y_2(-1) + 1 \biggr]
                &=&
                4
                \\
		\biggl[p_{-1}(-1)x_1(-1) \,+\, q_{-1}(-1) u_1(-1) \biggr] \times &&
		\nonumber\\
		\biggl[p_{+1}(-1)y_1(-1) \,+\, q_{+1}(-1)v_1(-1)\biggr]
                &=&
                0
	\label{eq.csp.cond.pqgen.dperm}
        \end{subeqnarray}
        Then the statistic $\Stat{}$ exhibits the cyclic sieving phenomenon with respect to involutions having $2^n F_{2n-1}$ fixed points.

\end{itemize}
\label{prop.dperm.csp}
\end{proposition}

\begin{proof}
We perform these substitutions in \cite[Theorems~3.3 and~3.9]{Deb-Sokal_genocchi}, respectively,
and in both cases obtain the rational function in~\eqref{eq.thm.fixedpoint.GC} substituted to $\we=\wo=1$.
\end{proof}


As a consequence of Proposition~\ref{prop.dperm.csp},
we can state our main theorem for D-permutations as an easy corollary:

\begin{theorem}[CSP for D-permutations] 
The following statistics on D-permutations 
exhibit the cyclic sieving phenomenon 
on D-permutations
with respect to involutions having $2^n F_{2n-1}$ fixed points
	(in particular, with respect to the Genocchi--Corteel involution on D-permutations):

\begin{quote}
neither-record-antirecord cycle peaks,
neither-record-antirecord cycle valleys,
neither-record-antirecords which are not fixed points,

lower crossing of type cpeak, lower nesting of type cpeak,
upper crossing of type cval, upper nesting of type cval,
lower crossings, lower nestings,
upper crossings, upper nestings,
crossings, nestings
\end{quote}

\label{thm.dperm.csp}
\end{theorem}

\subsection{Instances of the cyclic sieving phenomenon on D-semiderangements and D-derangements}\label{subsec.csp.dperm.sub}

We will now state instances of the cyclic sieving phenomena on D-e-semiderangements, D-o-semiderangements, and
D-derangements with respect to the respective restrictions of the Genocchi--Corteel involution.

Analogous to Proposition~\ref{prop.dperm.csp}, we have the following very general proposition, but, we only state the
$p,q$-versions:

\begin{proposition}\mbox{}\\
\begin{itemize}
	\item[(a)] 
	Let $\Stat{} : \dperm_{2n}^{{\rm e}} \to \mathbb{N}$ be a statistic defined on D-e-semiderangements.
	Assume that $\Stat{}$ is obtained as a specialisation of the polynomial 
	of the polynomial $\left.P_n^{(2)}\right|_{\ze=\we=0}$
	defined in~\eqref{def.Pn.pq}.
	Further assume that equations~(\ref{eq.csp.cond.pqgen.dperm}b,d)
	along with the following equations are satisfied:
        \begin{subeqnarray}
		\wo(-1) \;=\; \so(-1) &=& 1\\
		x_2(-1) \,\times\, \biggl[y_2(-1)\,+\, 1\biggr]  &=& 2
        \end{subeqnarray}
	Then the statistic $\Stat{}$ exhibits the cyclic sieving phenomenon with respect to involutions having $3^{n-1}$ fixed points.

	\item[(b)] 
	Let $\Stat{} : \dperm_{2n}^{{\rm eo}} \to \mathbb{N}$ be a statistic defined on D-derangements.
	Assume that $\Stat{}$ is obtained as a specialisation of the polynomial
        of the polynomial $\left.P_n^{(2)}\right|_{\ze=\we=\zo=\wo=0}$
        defined in~\eqref{def.Pn.pq}.
	Further assume that the equations~(\ref{eq.csp.cond.pqgen.dperm}b,d) along with
	\be
	x_2(-1) \times y_2(-1) \;=\;1
	\ee
	are satisfied.
        Then the statistic $\Stat{}$ exhibits the cyclic sieving phenomenon with respect to involutions having $2^{n-1}$ fixed points.

\end{itemize}
\label{prop.csp.dperm.subclass}
\end{proposition}

Finally, as a corollary we immediately obtain the following theorem:

\begin{theorem}
All statistics stated in Theorem~\ref{thm.dperm.csp} also exhibit the cyclic sieving phenomenon on
D-e-semiderangements, D-o-semiderangements, and D-derangements,
respect to involutions having $1$ fixed point when $n=0$, and
$3^{n-1}, 3^{n-1}, 2^{n-1}$ fixed points, respectively, for $n\geq 1$
(in particular, with the respective Genocchi–Corteel involutions on these classes of permutations.)
\end{theorem}

\section*{Acknowledgements}
We would like to thank Ashleigh Adams who informed us of this conjecture during AlCoVE 2024 
where this work was initiated.
We would also like to thank Jessica Striker for fruitful discussions and advice on this article.
We also thank Martin Rubey and Christian Stump for helping with updating the description of Statistic~123.

This work was supported by the DIMERS project ANR-18-CE40-0033 funded by Agence Nationale de la Recherche
(ANR, France).
The author is currently supported by the Tsinghua University Shuimu scholarship.

\appendix


\section{Positivity conjectures on distribution of vincular patterns}
\label{app.positivity.vincular}

Several sequences of numbers and polynomials associated to counting permutation patterns, or pattern-avoiding permutations have been 
conjectured to be Stieltjes moment sequences, see \cite[{Section~4.2, Open question}]{Bostan_20}, \cite{Blitvic_23}.
In this appendix, we introduce two new conjectures that adds to this growing list of ``positivity questions''
associated with permutation patterns.
Our conjectures here involve the distribution of occurences of vincular patterns of the form $a-bc$ and $bc-a$.

We first recall  \cite[Proposition~1]{Claesson_01} as stated in \cite[Proposition~2]{Claesson_02}:

\begin{proposition}
With respect to being equidistributed, the twelve vincular patterns of the form $a-bc$ or $bc-a$
fall into the three classes
\[\{1-23, 3-21, 12-3, 32-1\}\]
\[\{1-32, 3-12, 21-3, 23-1\}\]
\[\{2-13, 2-31, 13-2, 31-2\}\]
\end{proposition}
Following Claesson and Mansour \cite{Claesson_02}, we refer to these classes as Class 1, 2 and 3, respectively.
Finally, let $P_n^{(1)}(x), P_n^{(2)}(x)$, and $P_n^{(3)}(x)$ denote the generating polynomials of these three classes, respectively.

We have already noted in Theorem~\ref{thm.perm.vincular} 
that the generating function $\sum_{n=0}^\infty P_n^{(3)}(x) t^n$ 
can be written as a Stieltjes-type continued fraction. 
However, it seems that the generating function of the polynomial sequences
$\left(P_n^{(1)}(x)\right)_{n\geq 0}, \left(P_n^{(2)}(x)\right)_{n\geq 0}$
do not have any nice continued fraction.
A consequence of this continued fraction is that the sequence of polynomials $(P_n^{(3)}(x))_{n\geq 0}$
forms a Stieltjes moment sequence for $x\geq 0$.
Hence, it is natural to ask if the same is true for the polynomial sequences 
$\left(P_n^{(1)}(x)\right)_{n\geq 0}, \left(P_n^{(2)}(x)\right)_{n\geq 0}$
even though they do not have a nice continued fraction.
We have the following conjecture:

\begin{conjecture}
The sequences of polynomials  $\left(P_n^{(1)}(x)\right)_{n\geq 0}$, $\left(P_n^{(2)}(x)\right)_{n\geq 0}$,
where the polynomials $P_n^{(1)}(x)$ and $P_n^{(2)}(x)$
are the generating polynomials for the number of occurences of a vincular pattern of Class 1 or of Class 2, respectively,
are both Stieltjes moment sequences for $x\geq 0$.
\label{conj.vincular.sms}
\end{conjecture}

To verify this conjecture, for each of the two sequences we computed its S-fraction coefficients;
these are rational functions in the variable $x$.
We then used the \texttt{Reduce} function in Mathematica to compute the subdomain of $\mathbb{R}$ in which these are non-negative.
For the sequence $\left(P_n^{(1)}(x)\right)_{n\geq 0}$, we were able to do this for the first $30$ S-fraction coefficients,
and for the sequence $\left(P_n^{(2)}(x)\right)_{n\geq 0}$ we tested the first $37$ S-fraction coefficients.
In both cases, we found that these initial S-fraction coefficients are non-negative whenever $x\geq 0$.

A stronger consequence of the existence of the Stieltjes-type continued fraction for the sequence $(P_n^{(3)}(x))_{n\geq 0}$
is that it is coefficientwise-Hankel totally positive \cite[Theorem~9.9]{latpath_SRTR}.
Hence, we tested if the polynomial sequences $\left(P_n^{(1)}(x)\right)_{n\geq 0}, \left(P_n^{(2)}(x)\right)_{n\geq 0}$
are also coefficientwise Hankel totally positive.
Unfortunately, this is not the case.
However, this seems to be true after a shift $x\mapsto x+1$; we state this as a conjecture:


\begin{conjecture}
The sequences of polynomials  $\left(P_n^{(1)}(x+1)\right)_{n\geq 0}$, $\left(P_n^{(2)}(x+1)\right)_{n\geq 0}$,
where the polynomials $P_n^{(1)}(x)$ and $P_n^{(2)}(x)$
are the generating polynomials for the number of occurences of a vincular pattern of Class 1 or of Class 2, respectively, 
are both coefficientwise Hankel totally positive with respect to the variable $x$.
\label{conj.vincular.HTP}
\end{conjecture}

We have verified this conjecture for the first $10\times 10$ submatrix.
It is clear that 
Conjecture~\ref{conj.vincular.HTP} implies that Conjecture~\ref{conj.vincular.sms} holds for $x\geq 1$.


\end{document}